\documentclass{article}
\usepackage{graphicx} 
\usepackage{amsmath,amsfonts,amssymb,amsthm}
\usepackage{graphicx}
\usepackage{xcolor}
\usepackage{verbatim}
\usepackage[title]{appendix}
\usepackage{paralist}
\usepackage{url}
\usepackage{mathtools}
\newtheorem{theorem}{Theorem}
\newtheorem{lemma}{Lemma}

\theoremstyle{definition}
\newtheorem{remark}{Remark}
\newtheorem{proposition}{Proposition}

\renewcommand{\le}{\leqslant}
\renewcommand{\ge}{\geqslant}
\renewcommand{\leq}{\leqslant}
\renewcommand{\geq}{\geqslant}
\renewcommand{\emptyset}{\varnothing}

\newcommand{\cc}{\mathcal{C}} 
\newcommand{\ce}{\mathcal{E}}
\newcommand{\ck}{\mathcal{K}}
\newcommand{\cm}{\mathcal{M}}
\newcommand{\cz}{\mathcal{Z}}
\newcommand{\tmod}{\ \mathsf{mod}\ }

\newcommand{\real}{\mathbb{R}}
\newcommand{\ints}{\mathbb{Z}}
\newcommand{\natu}{\mathbb{N}}

\newcommand{\caln}{\mathcal{N}}

\newcommand{\bsa}{\boldsymbol{a}}

\newcommand{\bsk}{\boldsymbol{k}}

\newcommand{\bsq}{\boldsymbol{q}}

\newcommand{\bsx}{\boldsymbol{x}}
\newcommand{\bsy}{\boldsymbol{y}}

\newcommand{\bszero}{\boldsymbol{0}}
\newcommand{\bsone}{\boldsymbol{1}}
\newcommand{\bsell}{\boldsymbol{\ell}}

\newcommand{\bskappa}{\boldsymbol{\kappa}}

\newcommand{\rd}{\,\mathrm{d}}

\newcommand{\dunif}{\mathbb{U}}

\newcommand{\simiid}{\stackrel{\mathrm{iid}}\sim}

\newcommand{\eqd}{\stackrel{\mathrm{d}}=}

\newcommand{\nlb}{\nolinebreak}

\newcommand{\e}{\mathbb{E}}

\newcommand{\rank}{\mathrm{rank}}
\newcommand{\row}{\mathrm{Row}}

\newcommand{\one}{\mathbf{1}}

\newcommand{\norm}[1]{\Vert #1 \Vert_1}

\newcommand{\giv}{\!\mid\!}

\newcommand{\tran}{\mathsf{T}}
\newcommand{\phe}{\phantom{=}}

\newcommand{\walk}{\mathrm{wal}_k}
\newcommand{\walbsk}{\mathrm{wal}_{\bsk}}

\newcommand{\supp}{\boldsymbol{s}}

\newcommand{\normone}[1]{\Vert \lceil #1 \rceil \Vert_1}
\newcommand{\normtwo}[1]{\lceil#1\rceil_2}
\newcommand{\second}[1]{\lceil#1\rceil_{(2)}}

\title{Skewness of a randomized quasi-Monte Carlo estimate}
\author{Zexin Pan\\Stanford University
\and Art B. Owen\\Stanford University}
\date{Feb 2025}

\begin{document}

\maketitle
\begin{abstract}
Some recent work on confidence intervals for randomized quasi-Monte Carlo (RQMC) sampling found a surprising result: ordinary Student's $t$ 95\% confidence intervals based on a modest number of replicates were seen to be very effective and even more reliable than some bootstrap $t$ intervals that were expected to be best.  One potential explanation is that those RQMC estimates have small skewness. In this paper we give conditions under which the skewness is $O(n^\epsilon)$ for
any $\epsilon>0$, so `almost $O(1)$'.  Under a random generator matrix model, we can improve this rate to $O(n^{-1/2+\epsilon})$ with very high probability.  We also improve some probabilistic bounds on the distribution of the quality parameter $t$ 
for a digital net in a prime base under random sampling of generator matrices.
    
\end{abstract}
\section{Introduction}

We study the skewness of randomized quasi-Monte Carlo
(RQMC) sampling using some scrambled digital nets.
We are interested in this because it helps to
explain a phenomenon found in \cite{ci4rqmc}.  That was an
empirical investigation of confidence intervals
based on replicated RQMC estimates.  A standard Student's $t$ confidence
interval based on the central limit theorem
had quite good performance, and was much more reliable
than the bootstrap $t$ method that had been expected
to dominate. 

On further inspection, it emerged that the RQMC estimates tended to have modest or even very low skewness, although often with quite large kurtosis. 
The standard intervals are sensitive to skewness but
comparatively robust to kurtosis.  One significant advantage
of the bootstrap $t$ method is that it can handle
skewness well, but that advantage was not helpful
in the setting of \cite{ci4rqmc}.  In this paper
we consider the skewness of some scrambled net 
RQMC estimates for
$s$-dimensional integrands with continuous mixed partial derivatives
of order up to 2 in each variable. We believe that these are the first 
theoretical results on the
skewness of RQMC estimates.


Our approach is based on a Walsh function expansion
of the integrand from \cite{dick:pill:2010}.
As described in more detail later, that expansion
leads to an RQMC integration error that is a sum
of contributions that are Walsh coefficients multiplied
by an element of $\{-1,0,1\}$. Each such contribution
has a symmetric distribution because the probability of
multiplying by $+1$ equals the probability of 
multiplying by $-1$.  The integration error is then
a sum of symmetric random variables.  Those random
variables are uncorrelated.  A sum of independent
symmetric random variables is symmetric but such
need not hold when they are merely uncorrelated.
The primary technical difficulty arises from this dependence.

This paper is organized as follows.  Section~\ref{sec:back}
introduces our notation and some background material on the
construction of randomized digital nets, and their analysis
via a Walsh-Fourier decomposition.
Section~\ref{sec:skew} introduces the skewness of our RQMC
estimate, shows how some sums over triples of Walsh-Fourier
terms can be reduced to sums over pairs and bounds the
probability that such pairs contribute to the skewness.
Those are the results we later use to show
that the skewness is $O(n^\epsilon)$.
Section~\ref{sec:rand} considers generator matrices 
constructed completely at random and shows that 
with probability tending to one such matrices
yield skewness $O(n^{-1/2+\epsilon})$. Both skewness results assume that the variance is $\Omega(n^{-3})$ in a setting where it is known to be $O( n^{-3+\epsilon})$.
Of independent interest: we improve a probabilistic bound
on the quality parameter of digital nets in base $2$
compared to the result from \cite{larc:nied:schm:1996}
given as Theorem 5.37 of \cite{dick:pill:2010}.
Section~\ref{sec:example} computes the skewness of some simple test functions 
and those results match our theory.
Section~\ref{sec:disc} has a summary and adds some context
on the larger goal of forming confidence intervals
from RQMC estimates.

\section{Background and notation}\label{sec:back}
We use $\natu$ for the set $\{1,2,3,\dots\}$
of natural numbers and $\natu_0=\natu\cup\{0\}$,
and $\natu^s_* = \natu_0^s\setminus\{\bszero\}$.
For $L\in\natu$ we let
$\ints_L=\{0,1,\dots,L-1\}$. For $s\in\natu$
we write $1{:}s=\{1,2,\dots,s\}$.
For $\bsx\in\real^s$ and $u\subseteq1{:}s$ we use
$\bsx_u$ to denote the vector of elements $x_j$ for $j\in u$
and $\bsx_{-u}$ to denote the vector of elements $x_j$
for $j\not\in u$. We use $\bsone$ for the vector of $s$ ones.
The cardinality of a set $S$ is denoted by $|S|$.

For $a_n,b_n\in(-\infty,\infty)$ we write $a_n=O(b_n)$ if there exists $c<\infty$ and $N\in\natu$
with $|a_n|\le c|b_n|$ for all $n\ge N$.  In that case we also
write $b_n = \Omega(a_n)$. 

The integrands $f$ that we study are defined on $[0,1]^s$. We
will study them as members of $L^2[0,1]^s$. We estimate
$\mu = \int_{[0,1]^s}f(\bsx)\rd \bsx$ by 
$\hat\mu = (1/n)\sum_{i=0}^{n-1}f(\bsx_i)$
for points $\bsx_i\in[0,1]^s$ with $n=2^m$
for some $m\ge0$.  The $\bsx_i$ are obtained from
a digital construction described next.

\subsection{The point construction}
For $m\in\natu$ and $i\in \ints_{2^m}$ with
base $2$ expansion $i = \sum_{\ell=1}^{m}i_\ell 2^{\ell-1}$
we write $\vec{i}=\vec{i}[m] = (i_1,\dots,i_m)^\tran\in\{0,1\}^m$.
For $a\in[0,1)$ we can write $a=\sum_{\ell=1}^\infty a_\ell 2^{-\ell}$. Whenever there are two such expansions, we force
a unique choice by adopting the one with finitely many
$a_\ell=1$. Then for `precision' $E\ge1$ we write
$\vec{a}=\vec{a}[E]=(a_1,\dots,a_E)^\tran\in\{0,1\}^E$.

Our digital construction uses $s$ matrices $C_j\in\{0,1\}^{m\times m}$ for $m\ge1$. We assume at first that these matrices have full rank $m$ over $\ints_2$. Later we will consider random $C_j$ that may be rank deficient.
Unrandomized digital nets are defined via
\begin{align}\label{eq:plainqmc}
\vec{x}_{ij} = C_j\vec{i} \ \tmod 2
\end{align}
for $i\in\ints_{2^m}$ and $j\in1{:}s$.
Here $\vec{i}=\vec{i}[m]$ and
$x_{ij}$ is the point in $a\in[0,1)$
with $\vec{a}=\vec{a}[m] = \vec{x}_{ij}$. 
Finally,
$\bsx_i=(x_{i1},\dots,x_{is})$. We can include $m=0$
by convention, taking $\bsx_0=\bszero$.

The randomization we study comes from \cite{mato:1998:2}
and it has
\begin{align}\label{eqn:xequalMCiplusD}
\vec{x}_{ij} = M_jC_j\vec{i} + \vec{D}_j\ \tmod 2
\end{align}
for a random matrix $M_j\in \{0,1\}^{E\times m}$ and a random vector $\vec{D}_j\in\{0,1\}^E$ for precision $E\ge m$.
The matrices $M_j$ have ones on their diagonals,
zeros above their diagonals and $\dunif\{0,1\}$
random entries below their diagonals.
The vectors $\vec{D}_j$ consist of $\dunif\{0,1\}$
random entries.  All the random variables in all
of the $M_j$ and all of the $\vec{D}_j$ are 
mutually independent.  The matrices $M_j$ induce
a `matrix scramble' while the vectors $\vec{D}_j$
contribute a `digital shift'. The digital shift
makes each $\bsx_{i}$ uniformly distributed
over $\{0,1/2^E,2/2^E,\dots,(2^E-1)/2^E\}^s$ and for $E=\infty$,
each $\bsx_i\sim\dunif[0,1)^s$.

Letting $\bsx_i[E]$ denote the above points
constructed with precision $E$, our estimate is
$$
\hat\mu_E = \frac1n\sum_{i=0}^{n-1}f(\bsx_i)\quad\text{for $\bsx_i=\bsx_i[E]$}.$$
We focus our study on $\hat\mu_\infty$.
In practice, the value of $E$ is constrained
by the floating point representation in use.

\subsection{Elementary intervals and 
the $t$ parameter}

An elementary interval in base $2$ is
a Cartesian product of the form
$$
\mathrm{EI}(\bsk,\bsa) = \prod_{j=1}^s
\Bigl[\frac{a_j}{2^{k_j}},
\frac{a_j+1}{2^{k_j}}
\Bigr)
$$
for $\bsk\in\natu_0^s$ and $\bsa=(a_1,\dots,a_s)$ with each $a_j\in\ints_{2^{k_j}}$.
For $m\ge0$, the points $\bsx_0,\dots,\bsx_{2^m-1}$ defined at \eqref{eq:plainqmc}
are a $(t,m,s)$-net in base $2$ if
every elementary interval of volume $2^{t-m}$
contains precisely $2^t$ of those points.
Randomization does not change the value of
$t$.  See \cite{rtms} for a proof
with a different scrambling.
Smaller values of $t$ provide better
equidistribution and in
some cases one can obtain $t=0$.
We never have $t>m$ because
every $\bsx_i\in[0,1)^s$.

To construct a good digital one must make
a careful choice of matrices $C_1,\dots,C_s\in\{0,1\}^{m\times m}$.
The most popular choices are the Sobol' nets \cite{sobol67}. Sobol's construction requires one to choose some parameters called `direction numbers' and the ones from \cite{joe:kuo:2008} are widely used.

For $q\in\natu_0$ and $C\in\{0,1\}^{m\times m}$, let $C(1{:}q,:)\in\{0,1\}^{q\times m}$
denote the first $q$ rows of $C$.
Then for $q_1,\dots,q_s\in\natu_0$ let
$$
C^{\bsq}=C^{(q_1,\dots,q_s)}=
\begin{pmatrix}
C_1(1{:}q_1,:)\\
C_2(1{:}q_2,:)\\
\vdots\\
C_s(1{:}q_s,:)
\end{pmatrix}\in\{0,1\}^{\sum_{j=1}^sq_j\times m}.
$$
The matrices $C_1,\dots,C_s$ generate a $(t,m,s)$-net in base $2$
if $C^{\bsq}$ is full-rank whenever $q_1+\cdots+q_s= m-t$.
In this case the rank is taken over the finite field of $2$ elements,
that is, it is taken in arithmetic modulo $2$.

\subsection{Fourier-Walsh decomposition}

For $k\in\natu_0$ and $x\in[0,1)$, the $k$'th univariate
Walsh function $\walk$ is defined by
$$
\walk(x) = (-1)^{\vec{x}^\tran\vec{k}}.
$$
The inner product above has only finitely many
nonzero entries in it.  For $\bsk \in\natu_0^s$
$\walbsk:[0,1)^s\to\{-1,1\}$ is defined by
$$
\walbsk(\bsx) =\prod_{j=1}^s\mathrm{wal}_j(x_j)
=(-1)^{\sum_{j=1}^s\vec{k}_j^\tran \vec{x}_j}.
$$

The functions $\walbsk$ for $\bsk\in\natu_0^s$
form a complete orthonormal basis of $L^2[0,1)^s$~\cite{dick:pill:2010}.
We can then write
\begin{align*}
f(\bsx) &\sim \sum_{\bsk\in\natu_0^s}\hat f({\bsk}) \walbsk(\bsx),\quad\text{for}\\
\hat f({\bsk}) &= \int_{[0,1)^s}f(\bsx)\walbsk(\bsx)\rd\bsx.
\end{align*}
The Fourier-Walsh representation above for $f$ denotes
equality in mean square.

We will need some bounds on the decay of
Walsh coefficients.
The Walsh coefficients of smooth functions tend to
decay with increasing index, although in
quite different ways than Fourier coefficients
decay. For example, with $s=1$ and $r$
times differentiable $f$, a few coefficients
decay at the $O(1/k)$ rate while the majority
are much smaller. See  \cite{dick:2009} for these
results.

Our analysis of $\hat\mu$ requires several quantities
related to the indices $\bsk\in\natu_0^s$.  For each $k\in\natu_0$
we write $k = \sum_{\ell=1}^\infty a_\ell 2^{\ell-1}$ and then 
define the corresponding set $\kappa = \{\ell\in \natu\mid a_\ell=1\}$.
Any appearance of $k$ or $\kappa$ specifies the other one.
This $\kappa$ will select a set of rows of a scramble matrix.
The domain for $\kappa$ is $\caln=\{ K\subset \natu\mid |K|<\infty\}$.
The cardinality of $\kappa$ is $|\kappa|$ which is then the
number of nonzero bits in the binary expansion of $k$.
We use $\lceil \kappa\rceil=\max_{\ell\in\kappa}\ell$ to denote
the last element of $\kappa$ which determines the fineness of the
function $\walk$.  By convention, $\lceil\emptyset\rceil=0$.

Corresponding to $\bsk\in\natu_0^s$ there is
$\bskappa\in\caln^s$, with $\kappa_j$ corresponding to $k_j$
for each $j\in 1{:}s$. Now $\lceil\bskappa\rceil$ is the vector
with $j$'th component $\lceil\kappa_j\rceil$.
For $u\subseteq 1{:}s$ we use $\lceil \kappa_j,j\in u\rceil$
to denote a subvector of $\lceil\bskappa\rceil$
containing $\lceil\kappa_j\rceil$ for all $j\in u$.

In some settings we have several tuples of sets
such as $\bskappa_1$ and $\bskappa_2$.  Their $j$'th components
can be denoted via $\kappa_{1j}$ and $\kappa_{2j}$,
however for a specific value of $j$ such as $j=2$ we
use commas to make $\kappa_{1,2}$ more clearly distinct from $\bskappa_{12}$.

A theorem from \cite{yosh:2017} that we need uses
the sum of the largest $\alpha$ elements of a vector
or all of those elements if that vector has fewer than
$\alpha$ elements. In \cite{superpolymulti} we used
$\alpha=\infty$ which generates an $L^1$ norm. In the present setting we need $\alpha=2$.
For $\kappa\in\caln$ let
$$
\lceil \kappa\rceil_2 = 
\begin{cases}
0, & \kappa=\emptyset\\
\lceil\kappa\rceil, & |\kappa|=1\\
\max\{\ell_1+\ell_2\mid \ell_1,\ell_2\in\kappa, \ell_1\ne\ell_2\}, &\text{else.}
\end{cases}
$$
We use $\normtwo{\bskappa}$ for the vector with $j$'th
component $\normtwo{\kappa_j}$.
Note that the sum of the largest $q\le p$ absolute values of
the components of a vector $\bsx\in\real^p$
is a norm on $\real^p$ \cite{gran:boyd:ye:2006}.
We use $q=2$ with a convention for $p<2$.

In addition to the largest and sum of two largest elements, we also need notation for the second largest element. For this we use
$$
\second{\kappa}=\normtwo{\kappa}-\lceil\kappa\rceil
$$
with $\second{\kappa}=0$ when $|\kappa|\le1$.

We define some aggregate norms on $\bskappa$:
$$
\Vert\bskappa\Vert_0=\sum_{j=1}^s|\kappa_j|,
\quad
\normone{\bskappa}=\sum_{j=1}^s\lceil\kappa_j\rceil
\quad\text{and}\quad
\Vert\normtwo{\bskappa}\Vert_1=\sum_{j=1}^s\normtwo{\kappa_j}.
$$
The first is the total number of bits in all
the $k_j$, which is why the subscript is $0$
and not $1$. The next ones sum over $j$ the
indices of the highest one or two bits of $\kappa_j$.

We define `supports' of $\bskappa\in\caln^s$ and $\bsk\in\natu_0^s$ as
$$\supp(\bskappa)=\{j\in 1{:}s\mid \kappa_j\ne\emptyset\}
\quad\text{and}\quad \supp(\bsk) = \{j\in1{:}s\mid k_j>0\},$$
respectively, with $\supp(\bskappa)=\supp(\bsk)$. 
We have the following bound for Walsh coefficients
in terms of these quantities.

\begin{theorem}\label{thm:Walshcoefficientbound}
Let $f$ have continuous mixed partial derivatives up to order 2 in each variable $x_j$ on $[0,1]^s$. Then
\begin{align*}|\hat{f}(\bsk)|&\leq 2^{-\Vert\normtwo{\bskappa}\Vert_1-\Vert|\bskappa|\wedge2\Vert_1}\sup_{\bsx_{\supp(\bsk)}\in[0,1)^{|\supp(\bsk)|}}
\biggl|
\int_{[0,1]^{s-|\supp(\bsk)|}}
f^{|\bskappa|\wedge2}(\bsx)\rd\bsx_{-\supp(\bsk)}\biggr|
\end{align*}
where $\Vert |\bskappa|\wedge2\Vert_1=\sum_{j=1}^s\min(|\kappa_j|,2)$ and
\begin{align*}
f^{|\bskappa|\wedge2}
&
=\frac{\partial^{\sum_{j=1}^s\min(|\kappa_j|,2)}f}{\partial x_1^{\min(|\kappa_1|,2)}\cdots\partial x_s^{\min(|\kappa_s|,2)}}.
\end{align*}
\end{theorem}
\begin{proof}
This is from Theorem 2 of Yoshiki \cite{yosh:2017}
after taking $\alpha=2$ and $p=\infty$.
Our $\Vert\normtwo{\bskappa}\Vert_1+\Vert|\bskappa|\wedge2\Vert_1$
is his $\mu'_\alpha(\bsk_v)$ for $\alpha=2$ and $v=s(\bsk)$.
\end{proof}

We will not need the full strength of this result.
It will suffice for us to replace 
$2^{-\Vert\normtwo{\bskappa}\Vert_1-\Vert|\bskappa|\wedge2\Vert_1}$ by $2^{-\Vert\normtwo{\bskappa}\Vert_1}$.

The integration error satisfies
$$
\hat\mu_\infty-\mu
\sim\frac1n\sum_{i=0}^{n-1}
\sum_{\bsk\in\natu_0^s}\hat f(\bsk)\walbsk(\bsx_i)-\mu
=
\sum_{\bsk\in\natu_*^s}\hat f(\bsk)\frac1n\sum_{i=0}^{n-1}\walbsk(\bsx_i).
$$
Next we present a result on the
coefficient of $\hat f(\bsk)$ from
\cite{superpolymulti}.
For $\bsk\in\natu_0^s$, let $\cz(\bsk)$ be the event that $\sum_{j=1}^s \vec{k{}}_j^\tran M_j C_j=\bszero \tmod 2$. Similarly, let $Z(\bsk)$ be the indicator function $\one_{\cz(\bsk)}$.
We also use
$$
S(\bsk) = (-1)^{\sum_{j=1}^s\vec{k}_j^\tran \vec{D}_j}$$
the `sign of $\bsk$'.
We have the following mean square representation
of $f$.

\begin{theorem}\label{thm:decomp}
Let $f$ satisfy the condition in Theorem~\ref{thm:Walshcoefficientbound}  and let $\bsx_i$ be defined
by~\eqref{eqn:xequalMCiplusD} for $0\le i<2^m$.
Then
\begin{equation}\label{eqn:errordecomposition}
    \hat{\mu}_{\infty}-\mu\sim\sum_{\bsk\in \natu_*^s}Z(\bsk)S(\bsk)\hat{f}(\bsk).
\end{equation}
\end{theorem}
\begin{proof}
This is from \cite{superpolymulti}.    
\end{proof}

\section{Skewness}\label{sec:skew}

We have seen empirically that for large
$m$, the RQMC error $\hat\mu-\mu$ tends to
have a nearly symmetric distribution \cite{superpolymulti}.  Here
we study the skewness of $\hat\mu=\hat\mu_\infty$.
Without loss of generality we take $\mu=0$, so $\hat f(\bszero)=0$.
Then the skewness of $\hat\mu$ is
$$
\gamma = \e( \hat\mu^3 )/\e(\hat\mu^2)^{3/2}.
$$
Accuracy of confidence intervals is also studied
using the kurtosis
$$
\theta=\e( \hat\mu^4 )/\e(\hat\mu^2)^2-3.
$$
This quantity is commonly denoted by $\kappa$ but
we reserve $\kappa$ for subsets of $\natu$.
Neither $\gamma$ nor $\theta$ is well defined if $\e(\hat\mu^2)=0$.   We assume from here on that $\e(\hat\mu^2)>0$. With this assumption, we are ruling out integrands that
can be integrated with zero error by scrambled nets.  Those are integrands with a finite Walsh function
expansion.  A smooth function with a finite Walsh expansion must equal $\mu$ almost everywhere, which then provides only a trivial integration problem.

For $\bsx_i\simiid \dunif[0,1]^s$ the skewness of
$\hat\mu$ is $\gamma(f)/\sqrt{n}$ where $\gamma(f)$
is the skewness of a single function value $f(\bsx_0)$,
and the kurtosis is $\theta(f)/n$ where $\theta(f)$
is the kurtosis of a single  value $f(\bsx_0)$.
The RQMC points we use are far from independent and
so that formula does not apply.  
For instance, an analysis in \cite[Section 3]{superpolyone} 
shows that for differentiable $f$, there
is generally an event $\ce$ of probability $\Omega(1/n)$
with $\e(\hat\mu^2\giv \ce)=\Omega(n^{-2})$.
That event alone gives $\e(\hat\mu^4)=\Omega(n^{-5})$
and when $\e(\hat\mu^2)=O(n^{-3+\epsilon})$ we get
$\theta = \Omega( n^{1+2\epsilon})$.
Because $\theta\to\infty$, we do not get a central limit theorem for
the scrambling of \cite{mato:1998:2}.

We will consider integrands $f$ that
have continuous partial derivatives of order
up to $2$ in each of the $s$ components of
$\bsx$ on $[0,1]^s$.  This makes those partial derivatives absolutely continuous,
hence bounded, and as a result
$\e(|\hat\mu|^3)<\infty$ which we need in
order to have $\e(\hat\mu^3)$ be well defined.

The contribution of $\hat f(\bsk)$ 
to the error in~\eqref{eqn:errordecomposition}
is the product of $S(\bsk)\in\{-1,1\}$
and $Z(\bsk)\in\{0,1\}$. From the way
that $\vec{D}_j$ are generated we
have $S(\bsk)\sim\dunif\{-1,1\}$ and
these signs are pairwise independent
(though not fully independent).
The collection $(S(\bsk)\mid\bsk\in\natu^s)$
is independent of $(Z(\bsk)\mid\bsk\in\natu^s)$.
As a result each $Z(\bsk)S(\bsk)\hat f(\bsk)$
has a symmetric distribution with atoms
at $0$, $\hat f(\bsk)$ and $-\hat f(\bsk)$.
The error contributions in~\eqref{eqn:errordecomposition} all have
symmetric distributions and they are
uncorrelated because for $\bsk\ne\bsk'$
$$
\e\bigl( Z(\bsk)S(\bsk)Z(\bsk')S(\bsk')\bigr)
=\e\bigl( Z(\bsk)Z(\bsk')\bigr)\e\bigl(S(\bsk)S(\bsk')\bigr)=0
$$
using pairwise independence of
the signs.  The sum itself need not have a symmetric
distribution. Indeed we have seen asymmetric empirical distributions for small $m$.  However for large $m$ we typically see
very symmetric histograms.  One example,
typical of what we have seen, is 
in the top right panel of Figure 1 in
\cite{ci4rqmc}.

We can write the numerator of $\gamma$
as 
\begin{align*}
    \e(\hat{\mu}^3)=\sum_{\bsk_1\in \natu_*^s}\sum_{\bsk_2\in \natu_*^s}\sum_{\bsk_3\in \natu_*^s}\hat{f}(\bsk_1)\hat{f}(\bsk_2)\hat{f}(\bsk_3)\e\biggl(\prod_{\ell=1}^3Z(\bsk_\ell)\biggr)\e\biggl(\prod_{\ell=1}^3S(\bsk_\ell)\biggr).
\end{align*}
Letting $\bsk_{\ell}=(k_{\ell,1},\dots,k_{\ell,s})$ for $\ell=1,2,3$
\begin{align}\label{eq:prodsign}
\prod_{\ell=1}^3S(\bsk_\ell)
= (-1)^{\sum_{j=1}^s\vec{D}_j^\tran\sum_{\ell=1}^3\vec{k}_{\ell,j}}.
\end{align}
The expression in~\eqref{eq:prodsign} can have nonzero expectation
only if $\sum_{\ell=1}^3 \vec{k}_{\ell,j}=\bszero$ for each $j$. In that case we can write 
$$\vec{k}_{3,j}=-\vec{k}_{1,j}-\vec{k}_{2,j}=\vec{k}_{1,j}+\vec{k}_{2,j},$$
with the last step following because
we work modulo $2$.
Let $\bsk_{12}\in\natu_0^s$ be the vector
for which $\vec{k}_{12,j}=\vec{k}_{1,j}+\vec{k}_{2,j}\tmod 2$ for $j\in1{:}s$.
When $\e(\prod_{\ell=1}^3S(\bsk_\ell))\ne0$
it is because $\bsk_3=\bsk_{12}$ and then $\Pr(\,\prod_{\ell=1}^3S(\bsk_\ell)=1)=1$. As a result we may write
\begin{align}\label{eq:numer}
    \e(\hat{\mu}^3)=\sum_{\bsk_1\in \natu_*^s}\sum_{\bsk_2\in \natu_*^s}\hat{f}(\bsk_1)\hat{f}(\bsk_2)\hat{f}(\bsk_{12})\e\bigl(Z(\bsk_1)Z(\bsk_2)\bigr)
\end{align}
after noting that $Z(\bsk_1)=Z(\bsk_2)=1$
implies that
$$\sum_{j=1}^s\vec{k{}}_{1,j}^\tran M_jC_j
=\sum_{j=1}^s\vec{k{}}_{2,j}^\tran M_jC_j=\bszero$$
which then yields $Z(\bsk_{12})=1$.



\begin{proposition}\label{prop:basic}
Let $A\in\{0,1\}^{K\times m}$ have rank $r$
over $\ints_2$
and let $\bsy\in\{0,1\}^K$.
Then the set of solutions $\bsx\in\{0,1\}^m$
to $A\bsx=\bsy \tmod 2$
has cardinality $0$ or $2^{m-r}$.
\end{proposition}
\begin{proof}
The null space of $A$ has dimension $m-r$ so $A\bsx=\bszero\tmod 2$
has $2^{m-r}$ solutions $\bsx$. If $A\bsx=\bsy\tmod 2$ has any solution $\bsx'$
then we get all solutions by adding to $\bsx'$ an element of that
null space (including $\bszero$ which gives us back $\bsx'$). 
Otherwise there are $0$ solutions.
\end{proof}

\begin{lemma}\label{lem:PrZbound}
If $C_1,\dots,C_s$ generate a digital $(t,m,s)$-net in base $2$ and $\bsk\in\natu_*^s$, then 
$$\Pr\big(\cz(\bsk)\big)\leq 2^{-m+t+s}.$$
\end{lemma}
\begin{proof}
This is Corollary 1 of \cite{superpolymulti}.
\end{proof}

Next we develop some results on the probability that $\cz(\bsk_1)$ and $\cz(\bsk_2)$ both occur.  A key quantity in this work is the set
\begin{align}\label{eq:jboth}
J_{1\cap2}:= \bigl\{j\in 1{:}s\mid \lceil\kappa_{1j}\rceil
=\lceil\kappa_{2j}\rceil\bigr\}.
\end{align}
We also write $M_{1\cap2}$ for $(M_j\mid j\in J_{1\cap2})$, the corresponding tuple of scramble matrices.
\begin{lemma}\label{lem:Z1Z2indep}
If $J_{1\cap2}=\emptyset$, then 
    $\cz(\bsk_1)$ and $\cz(\bsk_2)$
    are independent. Otherwise, they are independent conditionally on scrambles $M_j,j\in J_{1\cap 2}$, that is
\begin{align*}
\Pr\bigl(\cz(\bsk_1)\cap \cz(\bsk_2)\giv M_{1\cap2}\bigr)=\Pr\bigl(\cz(\bsk_1)\giv M_{1\cap 2}\bigr)\Pr\bigl(\cz(\bsk_2)\giv M_{1\cap 2}\bigr).
\end{align*}
\end{lemma}
\begin{proof}
If $J_{1\cap 2}=1{:}s$, then conditionally on $M_{1\cap2}$, the events $\cz(\bsk_1)$ and
$\cz(\bsk_2)$ are constants and hence trivially
independent.
Otherwise, the event $\cz(\bsk)$ can be rewritten as 
\begin{align}\label{eq:eventzk}
\sum_{j\notin J_{1\cap 2}} \vec{k{}}_{j}^\tran M_j C_j=-\sum_{j\in J_{1\cap 2}} \vec{k{}}_{j}^\tran M_j C_j,\end{align} 
where the right hand side is nonrandom conditionally on $M_{1\cap 2}$. It suffices to show that $\sum_{j\notin J_{1\cap 2}} \vec{k{}}_{1j}^\tran M_j C_j$ is independent of $\sum_{j\notin J_{1\cap 2}} \vec{k{}}_{2j}^\tran M_j C_j$ given $M_{1\cap2}$. 
That follows once we show $\vec{k{}}_{1j}^\tran M_{j}$ is conditionally independent of $\vec{k{}}_{2j^*}^\tran M_{j^*}$ for every pair $j,j^*\notin J_{1\cap 2}$. 
Suppose at first that $J_{1\cap2}\ne\emptyset$.
If $j\ne j^*$
then $M_j$, $M_{j^*}$ and $M_{1\cap2}$ are mutually independent, making $M_j$ and $M_{j^*}$ conditionally independent given $M_{1\cap2}$, and so we only need to prove it for $j=j^*$. Because $j\notin J_{1\cap 2}$, $\lceil\kappa_{1j}\rceil\neq \lceil \kappa_{2j}\rceil$. Assume without loss of generality that $\lceil\kappa_{1j}\rceil> \lceil \kappa_{2j}\rceil$. Then $\vec{k{}}_{2j}^\tran M_{j}$ is fully determined by the first $\lceil\kappa_{1j}\rceil-1$ rows of $M_j$ while 
$$\vec{k{}}_{1j}^\tran M_{j}=M(\lceil\kappa_{1j}\rceil,:)+\sum_{r\in \kappa_{1j},r<\lceil\kappa_{1j}\rceil }M(r,:)\overset{d}{=}M(\lceil\kappa_{1j}\rceil,:),$$
which does not depend on the first $\lceil\kappa_{1j}\rceil-1$ rows of $M_j$.
Let $F$ represent the first $\lceil\kappa_{1j}\rceil-1$ rows  of $M_j$. Then
\begin{align*}
\e( Z(\bsk_1)Z(\bsk_2)\giv M_{1\cap 2})
&=\e\bigl( \e(Z(\bsk_1)Z(\bsk_2)\giv  F,M_{1\cap 2})\giv M_{1\cap2}\bigr)\\
&= \e\bigl( Z(\bsk_1)\e(Z(\bsk_2)\giv M_{1\cap2})\giv M_{1\cap2}\bigr)\\
&=\Pr(\cz(\bsk_1)\giv M_{1\cap2})\Pr(\cz(\bsk_2)\giv M_{1\cap2})
\end{align*}
as required. Finally, if $J_{1\cap2}=\emptyset$, then the same argument goes through except that conditioning on  $M_{1\cap2}$ is not conditioning at all, and we find that $M_j$ and $M_{j^*}$ are simply independent.

\end{proof}

\begin{theorem}\label{cor:Z1Z2bound}
     For any pair of distinct $\bsk_1$ and $\bsk_2$, we have
     \begin{align}\label{eq:bound3by2}
\Pr\bigl(\cz(\bsk_1)\cap \cz(\bsk_2)\bigr)\leq 2^{-(3/2)(m-t-s)}.
\end{align}
If further $\sum_{j\notin J_{1\cap 2}}\lceil\bskappa_{1j}\rceil\geq m-t$ or $s\leq 2$, then 
\begin{align}\label{eq:Z1Zbound}
    \Pr\bigl(\cz(\bsk_1)\cap \cz(\bsk_2)\bigr)\leq 2^{-2(m-t-s)}.
\end{align}
\end{theorem}
\begin{proof}
As in the proof of Lemma~\ref{lem:Z1Z2indep} we use 
$J_{1\cap2}$ from~\eqref{eq:jboth}
and the corresponding scrambles are in $M_{1\cap 2}$. We only need to consider cases where $\Vert\lceil\bskappa_{1}\rceil\Vert_1\geq m-t$ because otherwise $\Pr\bigl(\cz(\bsk_1)\bigr)=0$ by the $(t,m,s)$-net property. 

Let $A=\{j\in1{:}s\mid \lceil\kappa_{1j}\rceil\neq \lceil \kappa_{2j}\rceil, \lceil\kappa_{1j}\rceil\geq 1 \}$, that is $A=J^c_{1\cap2}\cap s(\kappa_1)$.
We let $C^{\lceil\kappa_{1j},j\in A\rceil}$ denote the matrix
made up of the first $\lceil\kappa_{1j}\rceil$ rows of $M_j$ for $j\in A$.
Then $C^{\lceil\kappa_{1j}-1,\ j\in A\rceil}$
has the first $\lceil\kappa_{1j}\rceil-1$ rows of those $M_j$.


Because each $M_j$ is lower-triangular with ones on the diagonal and $\dunif\{0,1\}$ random entries below the diagonal, we see that $\vec{k{}}_{1j}^\tran M_j\eqd M_j(\lceil\kappa_j\rceil,:)$. From independence of $M_j$, $\sum_{j\in A}\vec{k{}}_{1j}^\tran M_j C_j
\eqd\sum_{j\in A} M_j(\lceil\kappa_j\rceil,:)C_j$. Using $M_j(\lceil\kappa_j\rceil,\lceil\kappa_j\rceil)=1$, we may write
$$\sum_{j\in A} M_j(\lceil\kappa_j\rceil,:)C_j
= \sum_{j\in A}\sum_{\ell=1}^{\lceil\kappa_j\rceil-1} M_j(\lceil\kappa_j\rceil,\ell)C_j(\ell,:) + \sum_{j\in A} C_j(\lceil\kappa_j\rceil,:).
$$ 
We notice the first term on the right hand side is uniformly distributed on the row space $C^{\lceil\kappa_{1j}-1,\ j\in A\rceil}$ while the second term is nonrandom. We therefore conclude that $\sum_{j\in A}\vec{k{}}_{1j}^\tran M_j C_j$ is uniformly distributed on an affine space of dimension $\rank(C^{\lceil\kappa_{1j}-1,\ j\in A\rceil})$ and
\begin{align*}
    \Pr\bigl(\cz(\bsk_1)\mid M_{1\cap 2}\bigr)&=\Pr\Biggl(\,\sum_{j\in A} \vec{k{}}_{j}^\tran M_j C_j=-\sum_{j\in J_{1\cap 2}} \vec{k{}}_{j}^\tran M_j C_j\biggm| M_{1\cap 2}\Biggr)\\
    &\leq 2^{-\rank(C^{\lceil\kappa_{1j}-1,\, j\in A\rceil})}
\end{align*}
where we have used Proposition~\ref{prop:basic} in the last inequality. Next, we apply Lemmas~\ref{lem:PrZbound} and~\ref{lem:Z1Z2indep} to get
\begin{align*}
    \Pr\bigl(\cz(\bsk_1)\cap \cz(\bsk_2)\bigr)
    &=\e\bigl(\Pr(\cz(\bsk_1)\giv M_{1\cap2})\Pr(\cz(\bsk_2)\giv M_{1\cap 2})\bigr)\\
    &\leq 2^{-\rank(C^{\lceil\kappa_{1j}-1,\, j\in A\rceil})}\, \e\bigl(\Pr\bigl(\cz(\bsk_2)\giv M_{1\cap 2}\bigr)\bigr)\\
    &= 2^{-\rank(C^{\lceil\kappa_{1j}-1,\, j\in A\rceil})}\Pr\bigl(\cz(\bsk_2)\bigr)\\
    &\leq 2^{-m+t+s-\rank(C^{\lceil\kappa_{1j}-1,\, j\in A\rceil})}.
\end{align*}
We can assume $\rank(C^{\lceil\kappa_{1j}-1,\,j\in A\rceil})\geq \rank(C^{\lceil\kappa_{1j}-1,\,j\in J_{1\cap 2}\rceil})$ because otherwise we can replace $\bsk_2$ by $\bsk_{12}$. Now
\begin{align*}
&\,\quad\rank(C^{\lceil\kappa_{1j}-1,\, j\in A\rceil})
+\rank(C^{\lceil\kappa_{1j}-1,\,j\in J_{1\cap 2}\rceil
})\\
&\geq \rank(C^{\lceil\kappa_{1}\rceil-\bsone})\\
&\geq \rank(C^{\lceil\kappa_{1}\rceil})-s
\end{align*}
and the $(t,m,s)$-net property ensures $\rank(C^{\lceil\kappa_{1}\rceil})\geq m-t$, 
so we know $\rank(C^{\lceil\kappa_{1j}-1,\, j\in A\rceil})\geq 2^{-(m-t-s)/2}$ and 
\begin{align*}
    \Pr\bigl(\cz(\bsk_1)\cap \cz(\bsk_2)\bigr)\leq 2^{-(3/2)(m-t-s)}.
\end{align*}

When $\sum_{j\in A}\lceil\bskappa_{1j}\rceil=\sum_{j\notin J_{1\cap 2}}\lceil\bskappa_{1j}\rceil\geq m-t$, $C^{\lceil\kappa_{1j},j\in A\rceil}$ has rank at least $m-t$ and 
$$\rank(C^{\lceil\kappa_{1j}-1,\, j\in A\rceil})\geq \rank(C^{\lceil\kappa_{1j},j\in A\rceil})-|A|\geq m-t-s.$$
Therefore in this case, 
\begin{align*}
    \Pr\bigl(\cz(\bsk_1)\cap \cz(\bsk_2)\bigr)\leq 2^{-2(m-t-s)}.
\end{align*}

To prove the above bound also holds when $s\leq 2$, it suffices to show we always have $J_{1\cap 2}=\emptyset$ after replacing $\bsk_1$ or $\bsk_2$ by $\bsk_{12}$ if necessary, as $J_{1\cap 2}=\emptyset$ implies $\sum_{j\notin J_{1\cap 2}}\lceil\bskappa_{1j}\rceil=\Vert\lceil\bskappa_{1}\rceil\Vert_1\geq m-t$. When $s=1, J_{1\cap 2}=\{1\}$ or when $s=2, J_{1\cap 2}=\{1,2\}$,  we simply replace $\bsk_2$ by $\bsk_{12}$ and enforce $J_{1\cap 2}=\emptyset$. When $s=2, J_{1\cap 2}=\{1\}$, we can assume $\lceil \kappa_{12}\rceil <\lceil \kappa_{22}\rceil$ because otherwise we can switch $\bsk_1$ and $\bsk_2$. Then after replacing $\bsk_2$ with $\bsk_{12}$, we have $\lceil \kappa_{11}\rceil >\lceil \kappa_{21}\rceil$ and $\lceil \kappa_{21}\rceil < \lceil \kappa_{22}\rceil$, so $J_{1\cap 2}=\emptyset$. The $s=2, J_{1\cap 2}=\{2\}$ case can be handled in a similar way. This concludes our discussion.
\end{proof}

In general, the $3/2$ coefficient in equation~\eqref{eq:bound3by2} cannot be improved when $s\geq 3$. For instance, suppose $s=3$ and $m$ is an even number. Consider $\kappa_{1,1}=\kappa_{2,1}=\kappa_{1,2}=\kappa_{2,3}=\{m/2\}$ and $\kappa_{1,3}=\kappa_{2,2}=\emptyset$, in which case $J_{1\cap 2}=\{1\}$. Suppose there exist row vectors $v_1,v_2,v_3\in\{0,1\}^r$ with $r=m/2$ such that $v_1(r)=v_2(r)=v_3(r)=1$ and
$$v_1 C_1(1{:}r,:)=v_2C_2(1{:}r,:)=v_3C_3(1{:}r,:)\ \tmod 2.$$
Then in this case $\cz(\bsk_1)\cap \cz(\bsk_2)$ occurs when $M_1(r,1{:}r)=v_1$, $M_2(r,1{:}r)=v_2$ and $M_3(r,1{:}r)=v_3$, which occurs with $2^{-(3/2)m+3}$ probability given each $M_j(r,1{:}r)$ contains $r-1$ random entries. See Section~\ref{sec:example} for a concrete construction of $C_1,C_2,C_3$ for which such set of $v_1,v_2,v_3$ exists.

\section{Random generator matrices}\label{sec:rand}

We would like a better bound than
the one in~\eqref{eq:bound3by2}.  We can show
that a better bound holds for most generator
matrices if we consider generator matrices
$C_j\in\{0,1\}^{m\times m}$ made up of independent uniformly distributed elements. We notice that the counterexample in the previous section had
a common vector in the row spaces of $C^{\lceil\kappa_{1j},j\in J_{1\cap 2}\rceil}$, $C^{\lceil\kappa_{1j},j\notin J_{1\cap 2}\rceil}$ and $C^{\lceil\kappa_{2j},j\notin J_{1\cap 2}\rceil}$. That is unlikely to happen when their total rank is much less than $2m$. More precisely, 
we will show for any fixed $\epsilon>0$, with high probability random $C_j$'s satisfy 
$$\Pr\bigl(\cz(\bsk_1)\cap \cz(\bsk_2) \bigr)\leq 2^{-2(1-\epsilon)m}$$
for all pairs of distinct nonzero $\bsk_1$ and $\bsk_2$.

With random $C_j$ and random $M_j$ we get
a random product $M_jC_j$, and we could wonder whether these two
randomizations can be replaced by one. First of all, removing $C_j$ is
the same as taking $C_j=I_m$. We cannot do for $s\ge2$ because
applying a random $M_j$ will preserve the $t$ value of taking $C_1=C_2=\dots =C_d=I_m$.
For $s\ge2$, that value is $t=m-1$, a very bad outcome. Second, the product $M_jC_j$ does not have the uniform distribution on $\{0,1\}^{M\times m}$, so replacing
$M_jC_j$ by a single uniform random element of $\{0,1\}^{M\times m}$ changes the distribution. To see
this, there is a $2^{-m}$ probability that the first column of $C_j$ is all zeros.  Then for $M>m$, the probability that the first column of $M_jC_j$ is all zeros cannot be as small as $2^{-M}$ which it must be under a uniform distribution on $\{0,1\}^{M\times m}$. 

We recognize that the widely used constructions of
digital nets are very specially constructed. So they might then
be quite atypical compared to random generator matrices.

First, we need to find the average behavior of $t$ and $\Pr\bigl(\cz(\bsk_1)\cap \cz(\bsk_2)\bigr)$ under such randomization. Random generator matrices are discussed in Section 5.4 of \cite{dick:pill:2010}. 
Below we will use $\cm$ and $\cc$ as shorthand for $(M_j, j=1{:}s)$ and $(C_j, j=1{:}s)$, respectively. We will also write $t$ as $t(\cc)$ and $\cz(\bsk)$ as $\cz(\bsk,\cc)$ to emphasize their dependence on $\cc$.  Our main interest is in digital nets in base $2$, but some of these findings generalize to prime bases which are of independent interest.

\begin{theorem}\label{thm:dp5.4}
For a prime number $p$, let all elements of $C_1,\dots,C_s\in\ints_p^{m\times m}$ be independent $\dunif(\ints_p)$ random variables and let $t(\cc)$ be
the smallest $t\in\natu_0$ for which they generate
a $(t,m,s)$-net in base $p$. Then
$$\Pr\bigl(t(\cc)\le \lceil (s-1)\log_p(m)-\log_p(1-\alpha)\rceil\bigr)\ge\alpha$$
holds for $0\le \alpha<1$.
\end{theorem}
\begin{proof}
This is Theorem 5.37 of \cite{dick:pill:2010}. They credit \cite{larc:nied:schm:1996}.
\end{proof}

We are able to improve this
result for general $p$ though we only
need it for $p=2$.
Then we provide a theorem to control $\Pr(\cz(\bsk_1,\cc)\cap\cz(\bsk_2,\cc))$.

\begin{theorem}\label{thm:tofC}
Let $p$ be a prime number and suppose that
    $C_j$'s are independently drawn from a uniform distribution on $\ints_p^{m\times m}$ for $m\ge p$. Let $t(\cc)$ be the smallest $t\in\natu_0$ for which they generate a $(t,m,s)$-net in base $p$. Then 
    \begin{align}\label{eq:tofC}
    \Pr\bigl(t(\cc)\geq s \log_p(m)\bigr)\leq \frac{1}{(s-1)!m}.
    \end{align}
\end{theorem}

\begin{proof} 
Any $q\le m$ row vectors drawn independently from a uniform distribution on $\ints_p^m$ are linearly independent over $\ints_p$ with probability $\prod_{j=m-q+1}^m (1-p^{-j})$. 
We can show by induction that the probability of $q$ linearly dependent
rows is bounded by $p^{-m+q}$.  That is
\begin{align}\label{eq:lindepbound}
1-\prod_{j=m-q+1}^m (1-p^{-j})\leq p^{-m+q}.
\end{align}
It holds for $q=1$. Next, if~\eqref{eq:lindepbound} holds for $q=q^*$, then for $q=q^*+1$
\begin{align*}
1-\prod_{j=m-q^*}^m (1-p^{-j})&\leq 1-(1-p^{-m+q^*})(1-p^{-m+q^*})\\
&<2p^{-m+q^*}\le p^{-m+q^*+1},
\end{align*}
so~\eqref{eq:lindepbound} holds in general.
For $\cc$ to generate a $(t,m,s)$-net, we need $C^{\lceil \bskappa\rceil}$ to have full rank for all $\bskappa$ with $\normone{\bskappa}=\sum_{j=1}^s\lceil\kappa_j\rceil=m-t$.

There are ${m-t+s-1\choose s-1}$ different $C^{\lceil \bskappa\rceil}$ satisfying $\normone{\bskappa}=m-t$ and each of them has $m-t$ rows. Therefore the probability that any one of them fails to have full rank is bounded by
$${m-t+s-1\choose s-1}p^{-m+(m-t)}
={m-t+s-1\choose s-1}p^{-t}.$$
Because $m\ge p$, the threshold $s\log_p(m)$ for $t$ in~\eqref{eq:tofC} is at least $s$. With $t\ge s$
$$
{m-t+s-1\choose s-1}p^{-t}\le \frac{m^{s-1}}{(s-1)!}p^{-t}.
$$
The above is further bounded by $((s-1)!m)^{-1}$ once we take $t$ to be greater than or equal to $s\log_p(m)$, establishing equation~\eqref{eq:tofC}.
\end{proof}

Theorem~\ref{thm:tofC} has a sharper
bound on $t(\cc)$ than Theorem~\ref{thm:dp5.4}.  To make the comparison, let
$\alpha=1-1/((s-1)!m)$. Then
\begin{align*}
&\Pr\bigl( t(\cc)<s\log_p(m)\bigr)\ge\alpha&&\text{(Theorem~\ref{thm:tofC}),\quad and}\\
&\Pr\bigl( t(\cc)\le \lceil s\log_p(m)+\log_p( (s-1)!)\rceil\bigr)\ge\alpha&&\text{(Theorem~\ref{thm:dp5.4})}
\end{align*}
so Theorem~\ref{thm:tofC} has a much smaller upper bound
for $t(\cc)$ when $s$ is large. It has a restriction that $m\ne 1$ (the claim holds trivially for $m=0$) but applications with $p=2$ generally use $m\ge2$.


\begin{theorem}\label{thm:pairs}
Let $C_1,\dots, C_s$ be independently drawn from a uniform distribution on $\{0,1\}^{m\times m}$. If
 $\bsk_1$ and $\bsk_2$ are distinct and nonzero 
 points in $\natu_0^s$ with $\max_j\lceil\kappa_{1j}\rceil\leq m$ and $\max_j\lceil\kappa_{2j}\rceil\leq m$, then
    \begin{align}\label{eq:pbothk}
    \Pr\bigl(\cz(\bsk_1,\cc)\cap \cz(\bsk_2,\cc)\bigr)= 2^{-2m}.
    \end{align}
\end{theorem}
\begin{proof}
First, notice that
\begin{align}\label{eq:wecanswap}
\cz(\bsk_1,\cc)\cap\cz(\bsk_2,\cc)=\cz(\bsk_1,\cc)\cap\cz(\bsk_{12},\cc)
=\cz(\bsk_2,\cc)\cap\cz(\bsk_{12},\cc).
\end{align}
Therefore,
by replacing $\bsk_2$ by $\bsk_{12}$ if necessary, we can find $j^*\in1{:}s$ such that $\lceil\kappa_{1j^*}\rceil\neq \lceil\kappa_{2j^*}\rceil$. 
We assume without loss of generality that $\lceil\kappa_{1j^*}\rceil> \lceil\kappa_{2j^*}\rceil$. Then 
\begin{align*}
   &\phantom{=}\,\Pr\bigl(\cz(\bsk_1,\cc)\cap \cz(\bsk_2,\cc)\bigm| \cm, \cc\setminus C_{j^*}(\lceil\kappa_{1j^*}\rceil,:)\bigr)\\
   &= Z(\bsk_2,\cc)\Pr\bigl(\cz(\bsk_1,\cc)\bigm| \cm, \cc\setminus C_{j^*}(\lceil\kappa_{1j^*}\rceil,:)\bigr)\\
   &= Z(\bsk_2,\cc)\Pr\biggl(\,\sum_{r\in\kappa_{1j^*}}M_{j^*}(r,:)^\tran C_{j^*}=-\hspace*{-.2cm}\sum_{j=1,j\neq j^*}^s \vec{k{}}_{1j^*}^\tran M_{j^*} C_{j^*} \Bigm| \cm, \cc\setminus C_{j^*}(\lceil\kappa_{1j^*}\rceil,:)\biggr)\\
   &= 2^{-m}Z(\bsk_2,\cc).
\end{align*}
The first equality follows because $\cz(\bsk_2,\cc)$ does not depend on $C_{j^*}(\lceil\kappa_{1j^*}\rceil,:)$ and the last equality follows because $\sum_{r\in\kappa_{1j^*}}M_{j^*}(r,:)^\tran C_{j^*}=C_{j^*}(\lceil\kappa_{1j^*}\rceil,:)$ 
plus vectors determined by $\cm$ and $\cc\setminus C_{j^*}(\lceil\kappa_{1j^*}\rceil,:)$. 

Taking expectations of the above identity, we get $$\Pr( \cz(\bsk_1,\cc)\cap\cz(\bsk_2,\cc))=2^{-m}\Pr( \cz(\bsk_2,\cc)).$$
We can apply the above argument again to get
$$\Pr\bigl(\cz(\bsk_2,\cc)\mid \cm, \cc\setminus C_{j^*}(\lceil\kappa_{2j^*}\rceil,:)\bigr)=2^{-m}.$$
Therefore $\Pr(\cz(\bsk_2,\cc))=2^{-m}$ and then \eqref{eq:pbothk} follows.
\end{proof}

\begin{theorem}\label{thm:Z1Z2bound}
Let $C_1,\dots, C_s$ be independently drawn from a uniform distribution on $\{0,1\}^{m\times m}$. For any $\epsilon>0$, we can find a threshold $\underline{m}_{\epsilon,s}$ depending on $\epsilon$ and $s$ such that when $m\geq \underline{m}_{\epsilon,s}$ 
    $$\Pr\bigl(\cz(\bsk_1,\cc)\cap \cz(\bsk_2,\cc)\giv \cc\bigr)\leq 2^{-2(1-\epsilon)m}$$
    for all pairs of distinct nonzero $\bsk_1$ and $\bsk_2$ 
    for all $\cc$ in a set $\cc^*$ with probability at least $1-((s-1)!m)^{-1}-m^{3s}2^{-2\epsilon m}$.
\end{theorem}
\begin{proof} 

We first prove the bound for pairs $\bsk_1,\bsk_2$ such that $\max_j\lceil\kappa_{1j}\rceil\geq m$ or $\max_j\lceil\kappa_{2j}\rceil\geq m$. By symmetry, we can assume that $\lceil\kappa_{1j^*}\rceil\geq m$ for some $j^*\in1{:}s$, and by replacing $\bsk_2$ with $\bsk_{12}$ if necessary (because of~\eqref{eq:wecanswap}), we can assume that $j^*\notin J_{1\cap 2}$. Theorem~\ref{cor:Z1Z2bound} then shows  that
$$\Pr\bigl(\cz(\bsk_1,\cc)\cap \cz(\bsk_2,\mathcal{C})\giv \cc\bigr)\leq 2^{-2(m-t(\cc)-s)}\leq 2^{-2(1-\epsilon)m}$$
for large enough $m$ when $t(\cc)< s\log_2(m)$. From Theorem~\ref{thm:tofC},
\begin{align}\label{eq:badevent1}\Pr( t(\cc)\ge s \log_2(m)) \le ((s-1)!m)^{-1}.
\end{align}

We proceed to bound the remaining pairs. First notice that $\cz(\bsk_1,\cc)\cap \cz(\bsk_2,\cc)$ can be rewritten as these two equalities:
\begin{align*}\sum_{j=1}^s \sum_{r\in \kappa_{1j}\cap \kappa_{2j}}M_j(r,:) C_j+\sum_{r\in \kappa_{1j}\setminus \kappa_{2j}}M_j(r,:)C_j&=\bszero,\quad\text{and}\\
 \sum_{j=1}^s \sum_{r\in \kappa_{1j}\cap \kappa_{2j}}M_j(r,:) C_j+\sum_{r\in \kappa_{2j}\setminus \kappa_{1j}}M_j(r,:)C_j&=\bszero.
\end{align*}
Now $\sum_{r\in A}M_j(r,:)\eqd M_j(\lceil A\rceil,:)$ and if $A\cap B=\emptyset$ then $\sum_{r\in A}M_j(r,:)$ is independent of $\sum_{r\in B}M_j(r,:)$. Therefore 
$\cz(\bsk_1,\cc)\cap \cz(\bsk_2,\cc)$ is equivalent to
\begin{align*} \sum_{j=1}^s M_j\bigl(\lceil\kappa_{1j}\cap \kappa_{2j}\rceil,:\bigr) C_j+M_j\bigl(\lceil\kappa_{1j}\setminus \kappa_{2j}\rceil,:\bigr)C_j&=\bszero,\quad\text{and}\\
\sum_{j=1}^s M_j\bigl(\lceil\kappa_{1j}\cap \kappa_{2j}\rceil,:\bigr) C_j+M_j\bigl(\lceil\kappa_{2j}\setminus \kappa_{1j}\rceil,:\bigr)C_j&=\bszero.
\end{align*}
This implies that $\Pr\bigl(\cz(\bsk_1,\cc)\cap \cz(\bsk_2,\cc)\giv \cc\bigr)$ only depends on $\bsk_1,\bsk_2$ through rows $\lceil \kappa_{1j}\cap \kappa_{2j}\rceil$, $\lceil \kappa_{1j}\setminus\kappa_{2j}\rceil$ and $\lceil\kappa_{2j}\setminus\kappa_{1j}\rceil$ of $M_j$ for $j=1{:}s$. Since the remaining pairs satisfy $\max_j\lceil\kappa_{1j}\rceil< m$ and $\max_j\lceil\kappa_{2j}\rceil< m$, each of $\lceil \kappa_{1j}\cap \kappa_{2j}\rceil$, $\lceil \kappa_{1j}\setminus\kappa_{2j}\rceil$ and $\lceil \kappa_{2j}\setminus\kappa_{1j}\rceil$ can only take values from $0$ to $m-1$, constituting at most $m^{3s}$ combinations. For each combination, we can apply Markov's inequality and get
\begin{align*}
    &\phantom{\leq}\,\Pr\Bigl(\Pr\bigl(\cz(\bsk_1,\cc)\cap \cz(\bsk_2,\cc)\giv \mathcal{C}\bigr)\geq
    2^{-2(1-\epsilon)m}\Bigr)\\
    &\leq  2^{2(1-\epsilon)m}\Pr\bigl(\cz(\bsk_1,\cc)\cap \cz(\bsk_2,\cc)\bigr)\\
    &=2^{-2\epsilon m}.
\end{align*}
After taking a union bound, we see that with probability at least $1-m^{3s}2^{-2\epsilon m}$ the desired bound holds for all the remaining pairs. 
The conclusion follows once we combine the failure probabilities from the preceding two cases. 
\end{proof}

\begin{remark}
    Although our focus is not on the dimensionality analysis, we note the threshold $\underline{m}_{\epsilon,s}$ grows super-linearly in $s/\epsilon$ because we need $m^{3s}2^{-2\epsilon m}<1$ for $m\geq \underline{m}_{\epsilon,s}$ in order for the statement of Theorem~\ref{thm:Z1Z2bound} to make sense. This translates to a super-exponential growth for the sample size $n=2^m$. One can improve this result by analyzing each coordinate projection separately. For each $u\subseteq 1{:}s$, the projection of $\bsx_i, i\in \ints_n$ onto coordinates $j\in u$ is a digital net with generating matrices $C_j,j\in u$. For each projection, we can set a different $\epsilon_u$ and replace $t(\mathcal{C})$ in our previous analysis by the projection-specific $t$-value $t^*_u(\mathcal{C})$ defined in \cite{nonzerogain}. The benefit is that when $f$ has a low effective dimension as defined in \cite{cafmowen}, we can allow large $\epsilon_u$ for high-dimensional projections because their corresponding ANOVA components $f_u$ have a negligible influence on our QMC estimate. A full analysis is beyond the scope of this paper and left for future research. 
\end{remark}

Theorem~\ref{thm:Z1Z2bound} shows that when $m$ is large then for most realizations of $\cc$, $\Pr\bigl(\cz(\bsk_1,\cc)\cap \cz(\bsk_2,\cc)\giv \cc\bigr)$ has a comparable bound to that of the $s\leq 2$ case. From now on, we assume that we have generated our $\cc$, either by a pre-chosen design or by randomization, and continue our analysis on $\e[(\hat{\mu}_{\infty}-\mu)^3]$. We will again suppress the dependence of $t$ and $\cz(\bsk)$ on $\cc$.

Next we provide a technical lemma that we need in the proof of
our main result.  There we will have to sum a triple product of the lead factor in the bound from Theorem~\ref{thm:Walshcoefficientbound} over a collection of Walsh coefficient indices.  This next result is the core of the proof of our main theorem. Figure~\ref{fig:placement} helps to visualize the arguments.

\begin{figure}
\centering
\begin{tabular}{|l|ccccccccccccccc|}
\multicolumn{8}{r}{$\,\ell_1$}\\
\hline
$\kappa_1$ &$\circ$&$\circ$&$\cdots$&$\circ$&$\circ$&$\circ$&$\bullet$&&&&&&&&\\
\hline
$\kappa_{12}$&$\circ$&$\circ$&$\cdots$&$\circ$&$\circ$&$\circ$&$\bullet$&$\circ$&$\circ$&$\circ$&$\circ$&$\cdots$&$\circ$&$\circ$&$\bullet$\\
\hline
$\kappa_2$&$\circ$&$\circ$&$\cdots$&$\circ$&$\circ$&$\circ$&&$\circ$&$\circ$&$\circ$&$\circ$&$\cdots$&$\circ$&$\circ$&$\bullet$\\
\hline
\multicolumn{16}{r}{$\ell_2$}\\
\end{tabular}
\caption{\label{fig:placement}
This is a sketch of the sets $\kappa_1$, $\kappa_2$, $\kappa_{12}$, when all are non-empty.
Solid circles mark indices that must be in those sets including $\ell_1=\lceil\kappa_1\rceil$ and $\ell_2=\lceil\kappa_2\rceil=\lceil\kappa_{12}\rceil$. Without loss of generality $\ell_1\in\kappa_{12}\setminus\kappa_2$.
Open circles in $\kappa_1$ and $\kappa_2$ indicate possible locations of $\ell_3=\lceil\kappa_1\rceil_{(2)}$ and $\ell_4=\lceil\kappa_2\rceil_{(2)}$, respectively.
Every index must be in exactly $0$ or $2$ of sets $\kappa_1$, $\kappa_2$, $\kappa_{12}$. 
  }
\end{figure}

\begin{lemma}\label{lem:boundT}
For $\ell_1,\ell_2,\ell_{12}\in\natu_0$, let  $K(\ell_1,\ell_2,\ell_{12})$ be the set of pairs $(\kappa_1,\kappa_2)\in\caln^2$ such that $\lceil\kappa_1\rceil=\ell_1$ and $\lceil\kappa_2\rceil=\ell_2$ and $\lceil\kappa_{12}\rceil=\ell_{12}$. 
Let $$T(\ell_{1},\ell_{2},\ell_{12})=\sum_{(\kappa_{1},\kappa_{2})\in K(\ell_{1},\ell_{2},\ell_{12})}2^{-\normtwo{\kappa_{1}}-\normtwo{\kappa_{2}}-\normtwo{\kappa_{12}}}.$$
Then
\begin{equation}\label{eq:boundT}
T(\ell_{1},\ell_{2},\ell_{12})
<2^{-\ell_1-\ell_2-\ell_{12}+2}.
\end{equation}
\end{lemma}

\begin{proof}
Recalling that $\second{\kappa}$ was defined as the second largest element of $\kappa$, defaulting to zero if $|\kappa|<2$.  First we consider the case where any of $\ell_1$, $\ell_2$, $\ell_{12}$ equals zero.  For that case we can take $\ell_1=0$ without loss of generality. We can assume that $\kappa_2=\kappa_{12}$, for otherwise $T=0$. Now $\ell=\second{\kappa_2}$ is an integer between $0$ and $\max(\ell_2-1,0)$ inclusive.  For any nonzero $\ell$ there are $2^{\ell-1}$ sets $\kappa_2$ with $\lceil\kappa_2\rceil=\ell_2$ and $\second{\kappa_2}=\ell$.  Then
\begin{align*}
T(\ell_1,\ell_2,\ell_{12})
&= \sum_{\kappa_2:\lceil\kappa_2\rceil=\ell_2}2^{-2\ell_2-2\second{\kappa_2}}
=2^{-2\ell_2}\sum_{\ell=0}^{\max(\ell_2-1,0)}2^{-2\ell}
2^{\max(0,\ell-1)}\\
&<2^{-2\ell_2}\bigl[1+\sum_{\ell=1}^\infty2^{-\ell+1}\bigr]
< 2^{-2\ell_2+2}=2^{-\ell_1-\ell_2-\ell_{12}+2}.
\end{align*}

With $\min(\ell_1,\ell_2,\ell_{12})>0$ the relation $\vec{k{}}_{12}=\vec{k{}}_{1}+\vec{k{}}_{2}$ now implies that two of $\ell_{1},\ell_{2},\ell_{12}$ are equal while the remaining one is smaller than the other two. 
By symmetry, we can permute the inputs $\ell_{1},\ell_{2},\ell_{12}$ without changing the value of $T(\ell_{1},\ell_{2},\ell_{12})$.  Therefore we assume without loss of generality that $1\le\ell_{1}<\ell_{2}=\ell_{12}$. Furthermore, one of $\kappa_2,\kappa_{12}$ must contain $\ell_1$ and then the other one cannot.  By switching $\kappa_2$ with $\kappa_{12}$ when $\ell_1\in \kappa_2$, we see that
\begin{align*}
T(\ell_{1},\ell_{2},\ell_{12})
&=2\sum_{\kappa_{1}:\lceil\kappa_1\rceil=\ell_1}\sum_{\kappa_{2}:\lceil\kappa_2\rceil=\ell_2, \ell_1\notin \kappa_2}2^{-\normtwo{\kappa_{1}}-\normtwo{\kappa_{2}}-\normtwo{\kappa_{12}}}\\
&=2^{-\ell_1-2\ell_2+1}\sum_{\kappa_{1}:\lceil\kappa_1\rceil=\ell_1}\sum_{\kappa_{2}:\lceil\kappa_2\rceil=\ell_2, \ell_1\notin \kappa_2}2^{-\second{\kappa_{1}}-\second{\kappa_{2}}-\second{\kappa_{12}}}.
\end{align*}

 Suppose that $\second{\kappa_1}=\ell_3<\ell_1$  and $\second{\kappa_2}=\ell_4<\ell_2$. Then $\second{\kappa_{12}}=\max(\ell_1,\ell_4)$. 
We need to get an upper bound for the number of pairs $(\kappa_1,\kappa_2)$ that satisfy $\lceil \kappa_1\rceil=\ell_1$, $\lceil\kappa_2\rceil_1=\ell_2$, $\second{\kappa_1}=\ell_3$, $\second{\kappa_2}=\ell_4$ and $\ell_1\not\in\kappa_2$. An index $\ell$ with $1\le \ell<\min(\ell_3,\ell_4)$ can be in both $\kappa_1$ and $\kappa_2$, or in neither of them, or in exactly one of them (plus $\kappa_{12}$). That provides $4^{\max(\min(\ell_3,\ell_4)-1,0)}$ possibilities. Next consider $\ell$ with $\min(\ell_3,\ell_4)<\ell<\max(\ell_3,\ell_4)$.    If $\ell_4<\ell<\ell_3$ then $\ell$ can belong to $\kappa_1$ (and $\kappa_{12})$ or not. If $\ell_3<\ell<\ell_4$ (but $\ell\ne\ell_1$) then $\ell$ can belong to $\kappa_2$ (and $\kappa_{12})$ or not. These combinations multiply the number of choices by at most $2^{\max(\ell_3,\ell_4)-\min(\ell_3,\ell_4)-1}$. 
If $\ell=\min(\ell_3,\ell_4)>0$, then there are two choices if $\ell_3\ne\ell_4$ and just one if $\ell_3=\ell_4$. If $\ell=\max(\ell_3,\ell_4)$, there is only one choice whether or not $\ell_3=\ell_4$. Finally, if $\ell>\max(\ell_3,\ell_4)$, then there is just one choice for $\ell$: if $\ell=\ell_2$, then $\ell\in\kappa_2$ but not $\kappa_1$, while if $\ell\ne \ell_2$, then $\ell$ is not in either of those sets.

Therefore in this setting with $\ell_1>0$
\begin{align}\label{eqn:Slll}
    &\phe\ T(\ell_{1},\ell_{2},\ell_{12})\nonumber\\
    &\leq 2^{-\ell_1-2\ell_2+1}\sum_{\ell_3=0}^{\ell_1-1}\sum_{\ell_4=0}^{\ell_2-1}2^{-\ell_3-\ell_4-\max(\ell_1,\ell_4)}4^{\max(\min(\ell_3,\ell_4)-1,0)}2^{\max(\ell_3,\ell_4)-\min(\ell_3,\ell_4)}\nonumber\\
    &= 2^{-\ell_1-2\ell_2+1}\sum_{\ell_3=0}^{\ell_1-1}\sum_{\ell_4=0}^{\ell_2-1}2^{-\max(\ell_1,\ell_4)}4^{-\one\{\min(\ell_3,\ell_4)\neq 0\}}\nonumber\\
    &<2^{-\ell_1-2\ell_2+1}\sum_{\ell_3=0}^{\ell_1-1}\biggl(\,\sum_{\ell_4=0}^{\ell_1}2^{-\ell_1}+\sum_{\ell_4=\ell_1+1}^\infty2^{-\ell_4}\biggr)\nonumber\\
    &=\ell_1(\ell_1+2)2^{-2\ell_1-2\ell_2+1}.
\end{align}
Recall that $\ell_1<\ell_2=\ell_{12}$. Because $\sup_{\ell_1\geq 1}\ell_1(\ell_1+2)2^{-\ell_1}=2$, we conclude that
$T(\ell_{1},\ell_{2},\ell_{12})< 2^{-\ell_1-\ell_2-\ell_{12}+2}$.
\end{proof}

\begin{theorem}\label{thm:main}
Let $f:[0,1]^s\to\real$ have continuous mixed partial derivatives up to order $2$ in each variable $x_j$ for $j\in1{:}s$. Then
$$\e((\hat\mu-\mu)^3)=O(n^{-9/2+\epsilon})$$
for any $\epsilon>0$.
If also $s\leq 2$ or the generator matrices $\mathcal{C}$ belongs to the set $\mathcal{C}^*$ specified in Theorem~\ref{thm:Z1Z2bound}, then
$$\e((\hat\mu-\mu)^3)=O(n^{-5+\epsilon}).$$
\end{theorem}
\begin{proof}
From Theorem~\ref{cor:Z1Z2bound} and Theorem~\ref{thm:Z1Z2bound}, we introduce the bound 
\begin{equation}\label{eqn:Pr}
    \overline{P}
=\begin{cases}
2^{-2(m-t-s)}, & s\leq 2,\\
2^{-2(1-\epsilon)m}, & s\geq 3 \text{ and $\cc\in\cc^*$ of Theorem~\ref{thm:Z1Z2bound}},\\
2^{-(3/2)(m-t-s)}, &\text{else.}
\end{cases}
\end{equation}
By the $(t,m,s)$-net property, we find $\Pr\bigl(\cz(\bsk_1)\cap \cz(\bsk_2)\bigr)=0$ if $\normone{\bskappa_\ell}\leq m-t$ for any of $\ell\in\{1,2,12\}$. 
Applying the above bounds to~\eqref{eq:numer}, we see 
that $\e[(\hat{\mu}_{\infty}-\mu)^3]$ is at most
\begin{align}\label{eqn:3rdmoment}
\overline{P}\sum_{\bsk_1\in \natu_*^s}\sum_{\bsk_2\in \natu_*^s}\bigl|\hat{f}(\bsk_1)\hat{f}(\bsk_2)\hat{f}(\bsk_{12})\bigr|\prod_{\ell\in\{1,2,12\}}\one\{\normone{\bskappa_\ell}>m-t\}.
\end{align}

Now the Walsh coefficient bound from Theorem~\ref{thm:Walshcoefficientbound} gives
\begin{align*}
    |\hat{f}(\bsk_1)\hat{f}(\bsk_2)\hat{f}(\bsk_{12})|&\leq A  2^{-\Vert\normtwo{\bskappa_1}\Vert_1-\Vert\normtwo{\bskappa_2}\Vert_1-\Vert\normtwo{\bskappa_{12}}\Vert_1}\\
    &=A \prod_{j=1}^s 2^{-\normtwo{\bskappa_{1j}}-\normtwo{\bskappa_{2j}}-\normtwo{\bskappa_{12j}}}
\end{align*}
where 
$$A=\biggl(\,\sup_{\bskappa:|\bskappa|\in \{0,1,2\}^s}\sup_{\bsx_{\supp(\bsk)}\in[0,1)^{|\supp(\bsk)|}}\Big|\int_{[0,1)^{s-|\supp(\bsk)|}}f^{|\bskappa|\wedge2}(\bsx)\rd\bsx_{-\supp(\bsk)}\Big|\biggr)^3.$$

For $\bsell_1,\bsell_2,\bsell_{12}\in\natu_0^s$ let $\ck(\bsell_1,\bsell_2,\bsell_{12})$ be the set of pairs $(\bsk_1,\bsk_2)$ such that $\lceil\bskappa_1\rceil=\bsell_1$, $\lceil\bskappa_2\rceil=\bsell_2$ and $\lceil\bskappa_{12}\rceil=\bsell_{12}$. Then for any chosen $\bsell_1, \bsell_2, \bsell_{12}$ 
\begin{align}\label{eqn:fff}
&\phantom{\leq\,}    \sum_{(\bsk_1,\bsk_2)\in \ck(\bsell_1,\bsell_2,\bsell_{12})}|\hat{f}(\bsk_1)\hat{f}(\bsk_2)\hat{f}(\bsk_{12})|\nonumber\\
    &\leq\sum_{(\bsk_1,\bsk_2)\in \ck(\bsell_1,\bsell_2,\bsell_{12})}A \prod_{j=1}^s 2^{-\normtwo{\kappa_{1j}}-\normtwo{\kappa_{2j}}-\normtwo{\kappa_{12j}}}\nonumber\\
    &=A\prod_{j=1}^s \sum_{(\kappa_{1j},\kappa_{2j})\in K(\ell_{1j},\ell_{2j},\ell_{12j})}2^{-\normtwo{\kappa_{1j}}-\normtwo{\kappa_{2j}}-\normtwo{\kappa_{12j}}}
\end{align}
where $K(\ell_1,\ell_2,\ell_{12})$ is the set of pairs $(k_1,k_2)$ such that $\lceil \kappa_1\rceil=k_1$, $\lceil \kappa_2\rceil=k_2$ and $\lceil\kappa_{12}\rceil=k_{12}$.

Applying Lemma~\ref{lem:boundT} to
equation~\eqref{eqn:fff}, we get
\begin{align*}
    \sum_{(\bsk_1,\bsk_2)\in \ck(\bsell_1,\bsell_2,\bsell_{12})}|\hat{f}(\bsk_1)\hat{f}(\bsk_2)\hat{f}(\bsk_{12})|
    &\leq A\prod_{j=1}^s2^{-\ell_{1j}-\ell_{2j}-\ell_{12j}+2}\\
    &= A 2^{-\norm{\bsell_1}-\norm{\bsell_2}-\norm{\bsell_{12}}+2s}.
\end{align*}
Finally, we put the above bound into equation~\eqref{eqn:3rdmoment} and get
\begin{align}\label{eqn:cubicbound}
    \e[(\hat{\mu}_{\infty}-\mu)^3]&\leq \overline{P}\sum_{\bsell_1\in \natu_0^s}\sum_{\bsell_2\in \natu_0^s}\sum_{\bsell_{12}\in \natu_0^s}A 2^{-\norm{\bsell_1}-\norm{\bsell_2}-\norm{\bsell_{12}}+2s}\one\{\norm{\bsell_i}>m-t,i=1,2,12\}\nonumber\\
    &=\overline{P} A2^{2s}\Biggl(\sum_{\bsell\in \natu_0^s}2^{-\norm{\bsell}}\one\{\norm{\bsell}>m-t\}\Biggr)^3.
\end{align}

Because there are ${N+s-1\choose s-1}$ vectors $\bsell\in\natu_0^s$ with $\norm{\bsell}=N$,
\begin{align}\label{eqn:suml}
\begin{split}
    \sum_{\bsell\in \natu_0^s}2^{-\norm{\bsell}}\one\{\norm{\bsell}>m-t\}&=\sum_{N=m-t+1}^\infty {N+s-1\choose s-1}2^{-N}\\
    &=O(m^{s-1} 2^{-m+t})
\end{split}
\end{align}
where $s$ is deemed as a constant in the big-O notation. Therefore
$$\e[(\hat{\mu}_{\infty}-\mu)^3]=O(\overline{P}\, m^{3s-3} 2^{-3(m-t)}).$$
In view of equation~\eqref{eqn:Pr}, $\e[(\hat{\mu}_{\infty}-\mu)^3]=O(n^{-5+\epsilon})$ for any $\epsilon>0$ if $s\leq 2$ or  $\cc\in \cc^*$ of Theorem~\ref{thm:Z1Z2bound}, and is $O(n^{-9/2+\epsilon})$ in general. 
\end{proof}

Under our smoothness conditions on $f$, the RQMC variance is $O(n^{-3+\epsilon})$
because that holds for the scramble in \cite{rtms}
by \cite{smoovar,localanti}, and because the scramble of \cite{mato:1998:2}
has the same variance \cite{YueHic02a}.
When the RQMC variance is $\Omega(n^{-3})$, then RQMC has $\gamma = O(n^{-1/2+\epsilon})$ in the former case
and $O(n^\epsilon)$ otherwise.

\begin{remark}
It is possible to improve the above bound in the special cases with $s=1$ or $s=2$.
When $s=1$, equations~\eqref{eqn:fff} and \eqref{eqn:Slll} give
\begin{align*}
&\phantom{\leq\,} \sum_{(\bsk_1,\bsk_2)\in \ck(\bsell_1,\bsell_2,\bsell_{12})}|\hat{f}(\bsk_1)\hat{f}(\bsk_2)\hat{f}(\bsk_{12})|\\
&    \leq A\times T(\ell_{1},\ell_{2},\ell_{12})\\
    &\leq A\max(1,\ell_1)(\ell_1+2)2^{-2\ell_1-2\ell_2+1}
\end{align*}
using~\eqref{eqn:Slll}, where we have replaced $\ell_1$ by $\max(1,\ell_1)$ in order to include the case with $\ell_1=0$ that was not included in~\eqref{eqn:Slll}.
The constraint $\norm{\bsell_i}>m-t,i=1,2,12$ becomes $\ell_2>\ell_1>m$, so equation~\eqref{eqn:3rdmoment} becomes
\begin{align*}
    \e[(\hat{\mu}_{\infty}-\mu)^3]&\leq \overline{P}\sum_{\ell_1>m}\sum_{\ell_2>\ell_1} A\max(1,\ell_1)(\ell_1+2)2^{-2\ell_1-2\ell_2+1},
\end{align*}
which is easily seen to be $O(n^{-6+\epsilon})$ for any $\epsilon>0$,
because $\overline{P}=2^{-2(m-t-s)}$.
Then for $\e((\hat\mu-\mu)^2)=\Omega(n^{-3})$ we get $\gamma = O(n^{-3/2+\epsilon})$.

When $s=2$, we let $L_j=\min_{i\in\{1,2,12\}}\lceil\kappa_{ij}\rceil$ and $H_j=\max_{i\in\{1,2,12\}}\lceil\kappa_{ij}\rceil$. 
Equations~\eqref{eqn:fff} and \eqref{eqn:Slll} give
$$    \sum_{(\bsk_1,\bsk_2)\in \ck(\bsell_1,\bsell_2,\bsell_{12})}|\hat{f}(\bsk_1)\hat{f}(\bsk_2)\hat{f}(\bsk_{12})|\leq A\prod_{j=1}^2\max(1,L_j)(L_j+2)\prod_{j=1}^22^{-2L_j-2H_j+1}.$$

We recall that $T(\ell_{1},\ell_{2},\ell_{12})$ is symmetric in its arguments and we used the order $\ell_1<\ell_2=\ell_{12}$ in Lemma~\ref{lem:boundT}. 
For 
two dimensions, we reconsider our order. We know that among $i=1,2,12$, either one or two of the three satisfies both $\lceil\kappa_{i1}\rceil=H_1$ and $\lceil\kappa_{i2}\rceil=H_2$. If two of them satisfy the above conditions, we can assume, after reordering, that $\lceil\kappa_{1,1}\rceil=L_1$, $\lceil\kappa_{2,1}\rceil=\lceil\kappa_{12,1}\rceil=H_1$ and  $\lceil\kappa_{1,2}\rceil=L_2, \lceil\kappa_{2,2}\rceil=\lceil\kappa_{12,2}\rceil=H_2$. The constraint $\norm{\bsell_i}>m-t,i=1,2,12$ implies that $H_1+H_2>L_1+L_2>m-t$, so 
$$\prod_{j=1}^22^{-2L_j-2H_j+1}=2^{-2(L_1+L_2)-2(H_1+H_2)+2}\leq 2^{-4(m-t)+2}.$$

Now suppose that only one of the three satisfies the those conditions. Then 
we can assume, after reordering, that $\lceil\kappa_{1,1}\rceil=L_1, \lceil\kappa_{2,1}\rceil=\lceil\kappa_{12,1}\rceil=H_1$, $\lceil\kappa_{2,2}\rceil=L_2$ and $\lceil\kappa_{1,2}\rceil=\lceil\kappa_{12,2}\rceil=H_2$. The constraint $\norm{\bsell_i}>m-t,i=1,2,12$ implies that $L_1+H_2>m-t$ and $L_2+H_1>m-t$, so
$$\prod_{j=1}^22^{-2L_j-2H_j+1}=2^{-2(L_1+H_2)-2(L_2+H_1)+2}\leq 2^{-4(m-t)+2}.$$

In both cases $\sum_{(\bsk_1,\bsk_2)\in \ck(\bsell_1,\bsell_2,\bsell_{12})}|\hat{f}(\bsk_1)\hat{f}(\bsk_2)\hat{f}(\bsk_{12})|$ is $O(n^{-4+\epsilon})$ when $\norm{\bsell_i}>m-t,i=1,2,12$ and a similar calculation as in equation~\eqref{eqn:cubicbound} and \eqref{eqn:suml} shows $\e[(\hat{\mu}_{\infty}-\mu)^3]=O(\overline{P} \,
n^{-4+\epsilon})$, which is $O(n^{-6+\epsilon})$ in view of equation~\eqref{eqn:Pr}. Again, with $\e((\hat\mu-\mu)^2)=\Omega(n^{-3})$, these translate into $\gamma = O(n^{-3/2+\epsilon})$.
\end{remark}

\section{Examples}\label{sec:example}

Because it is numerically difficult to estimate the skewness, we instead compute the skewness of some simple functions to verify our theory. 
First, we consider the one-dimensional function $f(x)= x^2$. It integrates to $1/3$ over $[0,1]$. Denoting $\mu=1/3$, equation~\eqref{eq:numer} becomes
\begin{align*}
    \e((\hat{\mu}-\mu)^3)=\sum_{k_1\in \natu_*}\sum_{\substack{k_2\in \natu_*\\k_2\neq k_1}}\hat{f}(k_1)\hat{f}(k_2)\hat{f}(k_{12})\e\bigl(Z(k_1)Z(k_2)\bigr).
\end{align*}
We have explicitly excluded the $k_1=k_2$ case from the summand because $\hat{f}(0)=\mu\neq 0$.

The Walsh coefficients of $f$ are given in Example 14.3 of \cite{dick:pill:2010}: for $k\in \natu_0$
\begin{equation}\label{eqn:Walshxsquare}
    \hat{f}(k)=\begin{cases} 1/3 & \text{ if } k=0,\\
    - 2^{-\ell-1} & \text{ if } k=2^{\ell-1},\quad \ell\geq 1,\\
     2^{-\ell_1-\ell_2-1} & \text{ if } k=2^{\ell_1-1}+2^{\ell_2-1},\quad \ell_1>\ell_2\geq 1, \\
    0 & \text{ otherwise.}    
    \end{cases}
\end{equation}
In our notation, this implies that $\hat{f}(k)=0$ if $|\kappa|>2$. In order for $\hat{f}(k_1)\hat{f}(k_2)\hat{f}(k_{12})\neq 0$ with $\vec{k}_{12}=\vec{k}_{1}+\vec{k}_{2}\tmod 2$, we must have $|\kappa_1|+|\kappa_2|+|\kappa_{12}|=0 \tmod 2$ and  $|\kappa_1|+|\kappa_2|+|\kappa_{12}| \leq 6$. We can also exclude the $|\kappa_1|+|\kappa_2|+|\kappa_{12}|=0$ and $|\kappa_1|+|\kappa_2|+|\kappa_{12}|=2$ cases because at least one of $\kappa_1$, $\kappa_2$ and $\kappa_{12}$ must be empty, which implies that either $k_1=k_2$ or one of $k_1,k_2$ equals $0$.

When $|\kappa_1|+|\kappa_2|+|\kappa_{12}|=4$, we first consider the case $|\kappa_1|=1$, $|\kappa_2|=1$ and $|\kappa_{12}|=2$. This requires $\kappa_1=\{\ell_1\}$, $\kappa_2=\{\ell_2\}$ and $\ell_1\neq \ell_2$. The corresponding $\kappa_{12}=\{\ell_1,\ell_2\}$. When the generating matrix $C_1$ is the $m\times m$ identity matrix, $\mathcal{Z}(k_1)$ is the event $M_1(\ell_1,:)=\bszero$, which occurs with probability $2^{-m}$ only for $\ell_1>m$. It is also easy to see that $\mathcal{Z}(k_1)$ and $\mathcal{Z}(k_2)$ are independent. Therefore, the sum of $\hat{f}(k_1)\hat{f}(k_2)\hat{f}(k_{12})\e\bigl(Z(k_1)Z(k_2)\bigr)$ over all cases with $|\kappa_1|=|\kappa_2|=1$ and $|\kappa_{12}|=2$ is given by
\begin{align*}
    &2^{-2m} \sum_{\ell_1=m+1}^\infty \sum_{\substack{\ell_2=m+1\\ \ell_2\neq \ell_1}}^\infty (-2^{-\ell_1-1})(-2^{-\ell_2-1}) 2^{-\ell_1-\ell_2-1} \\
    =& 2^{-2m-3} \Big(2\sum_{\ell_1=m+1}^\infty \sum_{\ell_2=\ell_1+1}^\infty  4^{-\ell_1-\ell_2} \Big)\\
    =& \frac{1}{45}2^{-6m-2}.
\end{align*}
The cases with $|\kappa_1|=|\kappa_{12}|=1$ and $|\kappa_2|=2$ can be handled similarly as can the cases with $|\kappa_2|=|\kappa_{12}|=1$ and $|\kappa_1|=2$.

When $|\kappa_1|+|\kappa_2|+|\kappa_{12}|=6$, we must have $\kappa_1=\{\ell_1,\ell_2\}$, $\kappa_2=\{\ell_1,\ell_3\}$ and $\kappa_{12}=\{\ell_2,\ell_3\}$ for distinct $\ell_1,\ell_2,\ell_3$. We assume for now $\ell_1<\ell_2<\ell_3$.  Again when the generating matrix $C_1$ is the identity matrix, then $\mathcal{Z}(k_1)$ is the event $M_1(\ell_1,:)+M_1(\ell_2,:)=\bszero \tmod 2$, which occurs with probability $2^{-m}$ only when $\ell_2>m$. It follows both $\ell_2$ and $\ell_3$ need to be greater than $m$ in order for $\mathcal{Z}(k_1)\cap \mathcal{Z}(k_2)$ to occur with a positive probability. Moreover, $\mathcal{Z}(k_1)\cap \mathcal{Z}(k_2)$ can be written as the event ${M_1(\ell_1,:)}={M_1(\ell_2,:)}={M_1(\ell_3,:)}$, which happens with $2^{-2m}$ probability when $\ell_3>\ell_2>m$. Therefore, the sum of $\hat{f}(k_1)\hat{f}(k_2)\hat{f}(k_{12})\e\bigl(Z(k_1)Z(k_2)\bigr)$ over all $\ell_1<\ell_2<\ell_3$ cases is given by
\begin{align*}
    &2^{-2m} \sum_{\ell_1=1}^\infty \sum_{\substack{\ell_2=m+1\\ \ell_2>\ell_1}}^\infty \sum_{\substack{\ell_3=m+1\\ \ell_3>\ell_2}}^\infty  2^{-\ell_1-\ell_2-1} 2^{-\ell_1-\ell_3-1} 2^{-\ell_2-\ell_3-1} \\
    =& 2^{-2m-3} \Big(\sum_{\ell_1=1}^{m} \sum_{\ell_2=m+1}^\infty \sum_{\ell_3=\ell_2+1}^\infty  4^{-\ell_1-\ell_2-\ell_3} + \sum_{\ell_1=m+1}^{\infty} \sum_{\ell_2=\ell_1+1}^\infty \sum_{\ell_3=\ell_2+1}^\infty  4^{-\ell_1-\ell_2-\ell_3}  \Big)\\
    =& 2^{-2m-3} \Big(\frac{1}{135} 2^{-4m} (1-2^{-2m})+\frac{1}{2835}2^{-6m}\Big)\\
    =&  \frac{1}{135} 2^{-6m-3} -\frac{20}{2835} 2^{-8m-3}.
\end{align*}
The cases corresponding to the other $5$ orderings of $\ell_1,\ell_2,\ell_3$ can be handled similarly.

Summing all cases together shows that
\begin{align*}
    \e( (\hat{\mu}-\mu)^3)=&\frac{3}{45}2^{-6m-2}+\frac{6}{135} 2^{-6m-3} -\frac{120}{2835} 2^{-8m-3}\\
    =& \Big(\frac{1}{45}-\frac{1}{189} 2^{-2m}\Big)2^{-6m}.
\end{align*}
A similar calculation shows that the variance of $\hat{\mu}$ can be lower bounded by
\begin{align}\label{eqn:varlowerbound}
    \e( (\hat{\mu}-\mu)^2) =&\sum_{k\in\natu_*}  \hat{f}(k)^2 \e\bigl(Z(k)\bigr) \nonumber\\
    \geq &\sum_{\ell=m+1}^\infty \hat{f}(2^{\ell-1})^2 \e\bigl(Z(2^{\ell-1})\bigr) \nonumber\\
    =& 2^{-m}\sum_{\ell=m+1}^\infty 4^{-\ell-1} \nonumber\\
    = &\frac{1}{12} 2^{-3m}
\end{align}
This shows that the skewness $\gamma$ is $O(2^{-(3/2)m})=O(n^{-3/2})$.

Next, we consider a three-dimensional example $g(\bsx)=\prod_{j=1}^3 x_j^2$. For $\bsk= (k_1,k_2,k_3)\in \natu^3_0$, its Walsh coefficients are given by $\hat{g}(\bsk)=\hat{f}(k_1)\hat{f}(k_2)\hat{f}(k_3)$ for $\hat{f}(k)$ defined by equation~\eqref{eqn:Walshxsquare}. Denoting $\mu=1/27$, equation~\eqref{eq:numer} becomes
\begin{align*}
    \e((\hat{\mu}-\mu)^3)=\sum_{\bsk_1\in \natu^3_*}\sum_{\substack{\bsk_2\in \natu^3_*\\\bsk_2\neq \bsk_1}}\hat{g}(\bsk_1)\hat{g}(\bsk_2)\hat{g}(\bsk_{12})\e\bigl(Z(\bsk_1)Z(\bsk_2)\bigr).
\end{align*}
Again we explicitly exclude the $\bsk_1=\bsk_2$ case because $\hat{g}(0)=\mu\neq 0$.

Proceeding as before, $\hat{g}(\bsk_1)\hat{g}(\bsk_2)\hat{g}(\bsk_{12})\neq 0$ with $\vec{k}_{12,j}=\vec{k}_{1,j}+\vec{k}_{2,j}\tmod 2$ for $j\in 1{:}3$ implies that $|\kappa_{1j}|+|\kappa_{2j}|+|\kappa_{12j}|=0 \tmod 2$ and  $|\kappa_{1j}|+|\kappa_{2j}|+|\kappa_{12j}| \leq 6$ for $j\in 1{:}3$. However, here we cannot remove the $|\kappa_{1j}|+|\kappa_{2j}|+|\kappa_{12j}|=0$ case or the $|\kappa_{1j}|+|\kappa_{2j}|+|\kappa_{12j}|=2$ case from consideration because they do not necessarily imply that $\bsk_1=\bsk_2$ or that one of $\bsk_1,\bsk_2$ equals $\bszero$. 

For simplicity, we will only analyze the case where $|\kappa_{1j}|+|\kappa_{2j}|+|\kappa_{12j}|=2$ for all $j\in 1{:}3$. The other cases can be handled in a similar way. Without loss of generality, we assume that $\bsk_1=(2^{\ell_1-1}$, $2^{\ell_2-1},0)$, $\bsk_2=(2^{\ell_1-1},0,2^{\ell_3-1})$ and $\bsk_{12}=(0,2^{\ell_2-1},2^{\ell_3-1})$ for $\ell_1,\ell_2,\ell_3\geq 1$. The corresponding $\hat{g}(\bsk_1)\hat{g}(\bsk_2)\hat{g}(\bsk_{12})$ is equal to
$$\Big(\frac{1}{3}2^{-\ell_1-\ell_2-2}\Big)\Big(\frac{1}{3}2^{-\ell_1-\ell_3-2}\Big)\Big(\frac{1}{3}2^{-\ell_2-\ell_3-2}\Big)=\frac{1}{27}2^{-2\ell_1-2\ell_2-2\ell_3-6}.$$
It remains to calculate $\e\bigl(Z(\bsk_1)Z(\bsk_2)\bigr)$.

For $C_1,C_2,C_3$ generating a $(t,m,3)$-net, $\mathcal{Z}(\bsk_1)$ and $\mathcal{Z}(\bsk_2)$ are the events $M_1(\ell_1,:)C_1+M_2(\ell_2,:)C_2=\bszero \tmod 2$ and  $M_1(\ell_1,:)C_1+M_3(\ell_3,:)C_3=\bszero \tmod 2$, respectively. Furthermore, $\mathcal{Z}(\bsk_1)\cap\mathcal{Z}(\bsk_2)$ can be written as the event $M_1(\ell_1,:)C_1=M_2(\ell_2,:)C_2=M_3(\ell_3,:)C_3 \tmod 2$. By the $(t,m,3)$-net property, $\mathcal{Z}(\bsk_1)\cap\mathcal{Z}(\bsk_2)$ occurs with positive probability only if $\ell_1+\ell_2> m-t$, $\ell_1+\ell_3> m-t$ and $\ell_2+\ell_3>m-t$. By Theorem~\ref{cor:Z1Z2bound}, we also know that $\Pr(\mathcal{Z}(\bsk_1)\cap\mathcal{Z}(\bsk_2))\leq 2^{-2(m-t-3)}$ if $\max(\ell_1,\ell_2,\ell_3)> m-t$, so we restrict our attention to cases with $\max(\ell_1,\ell_2,\ell_3)\leq m-t$. 

Without further assumptions on the generating matrices, we cannot bound $\Pr(\mathcal{Z}(\bsk_1)\cap\mathcal{Z}(\bsk_2))$ by $O(2^{-2m})$. Therefore the result that $\gamma = O(n^{-1/2+\epsilon})$ for most realizations of random generating matrices does not hold for all sequences of nonrandom generating matrices with some fixed $t$. We illustrate this by the following  example. Let $m$ be an even number and $r=m/2$. Suppose that $C^*_1,C^*_2,C^*_3$ are the generating matrices for a $(1,m-1,3)$-net (for instance Sobol' nets \cite{sobol67}). For each $j\in 1{:}3$, $C^*_j$ admits a block matrix representation
$$C^*_j=\begin{pmatrix}
    C^*_{j11} & C^*_{j12}\\
    C^*_{j21} & C^*_{j22}  
\end{pmatrix}$$
with $C^*_{j11}\in\{0,1\}^{(r-1)\times(r-1)}$ and $C^*_{j22}\in\{0,1\}^{r\times r}$. Then we can construct the following $m$ by $m$ matrix for $j\in 1{:}3$
$$C_j=\begin{pmatrix}
    C^*_{j11} & \bszero^T & C^*_{j12}\\
    \bszero & 1 & \bszero \\
    C^*_{j21} & \bszero^T & C^*_{j22}  
\end{pmatrix},$$
where $\bszero$'s are zero row vectors of suitable length to make the block matrix compatible. It is straightforward to check that $C_1,C_2,C_3$ generates a $(2,m,3)$-net. Moreover, for $\ell_1=\ell_2=\ell_3=r$, $M_1(\ell_1,:)C_1=M_2(\ell_2,:)C_2=M_3(\ell_3,:\nlb)C_3 \tmod 2$ if $M_j(r,1{:}(r-1))=\bszero$ for all $j\in 1{:}3$. Hence $\mathcal{Z}(\bsk_1)\cap\mathcal{Z}(\bsk_2)$ occurs with probability at least $2^{-3r+3}=2^{-(3/2)m+3}$ and the corresponding $\hat{g}(\bsk_1)\hat{g}(\bsk_2)\hat{g}(\bsk_{12})\e\bigl(Z(\bsk_1)Z(\bsk_2)\bigr)$ is at least $2^{-3r+3} (2^{-6r-6}/27)=2^{-(9/2)m-3}/27$. One cannot expect the skewness to converge to $0$ given such generating matrices, so we will add a condition.

To state our additional assumption, we define, for $C\in\{0,1\}^{m\times m}$ and $1\nlb\leq\nlb\ell\nlb\leq\nlb{m}$, $\row(C,\ell)$ to be the row space spanned by the first $\ell$ rows of $C$ over $\ints_2$. For generating matrices $C_1,C_2,C_3 \in\{0,1\}^{m\times m}$, we define
\begin{equation}\label{eqn:Tdef}
    T=\min\Bigl\{T'\geq 0 \bigm| \bigcap_{j=1}^3 \row(C_j,\ell_j) = \{\bszero\} \text{ whenever } \sum_{j=1}^3 \ell_j = 2m -T'\Bigr\}.
\end{equation}
Here again we assume that $1\leq \ell_j\leq m$ for all $j\in 1{:}s$. Our constructed example shows that $T$ can be unbounded even when $t\leq 2$.

Below we give a brief proof that $\Pr(\mathcal{Z}(\bsk_1)\cap\mathcal{Z}(\bsk_2))\leq 2^{-2m+T+3}$ when $\max(\ell_1,\ell_2,\ell_3)\leq m-t$. First, the $(t,m,3)$-net property guarantees that $C_j(1{:}\ell_1,:\nlb)$ has full rank for each $j\in 1{:}3$. This implies that every $w\in \bigcap_{j=1}^3 \row(C_j,\ell_j)$ has a unique representation as a linear sum of the first $\ell_j$ row vectors of $C_j$ for each $j\in 1{:}3$, which allows us to write 
$$w=v_1 C_1( 1{:}\ell_1,:)=v_2C_2( 1{:}\ell_2,:)=v_3C_3( 1{:}\ell_3,:)\tmod 2$$
with $v_j \in \{0,1\}^{\ell_j}$ for $j\in 1{:}3$. When $w=\bszero$, $v_1,v_2,v_3$ are all zero vectors, so it cannot occur that
$$\bszero=M_1(\ell_1,:)C_1=M_2(\ell_2,:)C_2=M_3(\ell_3,:)C_3 \tmod 2$$
given $M_j(\ell_j,:)$ is nonzero for $\ell_j\leq m$, $j\in 1{:}3$. For each nonzero $w\in \bigcap_{j=1}^3 \row(C_j,\ell_j)$, the event
$$w=M_1(\ell_1,:)C_1=M_2(\ell_2,:)C_2=M_3(\ell_3,:)C_3 \tmod 2$$
occurs with probability at most $2^{-\ell_1-\ell_2-\ell_3+3}$ since there are $\ell_j-1$ $\dunif\{0,1\}$ random entries in $M_j(\ell_j,:)$ for each $j\in 1{:}3$. A union bound over all nonzero $w\in \bigcap_{j=1}^3 \row(C_j,\ell_j)$ proves that
\begin{equation}\label{eqn:rowinsectbound}
  \Pr(\mathcal{Z}(\bsk_1)\cap\mathcal{Z}(\bsk_2))\leq \bigg(\Big|\bigcap_{j=1}^3 \row(C_j,\ell_j)\Big|-1\bigg) 2^{-\ell_1-\ell_2-\ell_3+3}.  
\end{equation}
Next, we notice that $\bigcap_{j=1}^3 \row(C_j,\ell_j)$ is also a vector space over $\ints_2$, so it has a rank $r(\ell_1,\ell_2,\ell_3)$ that depends on $\ell_1,\ell_2,\ell_3$ and its cardinality is $2^{r(\ell_1,\ell_2,\ell_3)}$. By the definition of $T$, $r(\ell_1,\ell_2,\ell_3)=0$ if $\ell_1+\ell_2+\ell_3\leq 2m-T$, in which case $\Pr(\mathcal{Z}(\bsk_1)\cap\mathcal{Z}(\bsk_2))=0$. When $\ell_1+\ell_2+\ell_3> 2m-T$, we can bound equation~\eqref{eqn:rowinsectbound} by
\begin{equation*}
  \Pr(\mathcal{Z}(\bsk_1)\cap\mathcal{Z}(\bsk_2))\leq 2^{r(\ell_1,\ell_2,\ell_3)-\ell_1-\ell_2-\ell_3+3}.  
\end{equation*}
It remains to show that $\ell_1+\ell_2+\ell_3-r(\ell_1,\ell_2,\ell_3)\geq 2m-T$ holds when $\ell_1+\ell_2+\ell_3> 2m-T$. Let us first compare $r(\ell_1,\ell_2,\ell_3)$ with $r(\ell_1+1,\ell_2,\ell_3)$. Notice that $\row(C_1,\ell_1+1)$ is the union of $\row(C_1,\ell_1)$ and $\row^{+}(C_1,\ell_1)$, the latter denoting the affine space of $\row(C_1,\ell_1)$ shifted by $C(\ell_1+1,:)$. Let $R_{23}=\row(C_2,\ell_2)\bigcap \row(C_3,\ell_3)$. We have
$$\Big|\row(C_1,\ell_1+1)\bigcap R_{23}\Big|\leq \Big|\row(C_1,\ell_1)\bigcap R_{23}\Big|+\Big|\row^{+}(C_1,\ell_1)\bigcap R_{23}\Big|.$$
If $\row^{+}(C_1,\ell_1)\bigcap R_{23}$ is empty, then $r(\ell_1+1,\ell_2,\ell_3)=r(\ell_1,\ell_2,\ell_3)$. Otherwise, we can find a vector $w^*\in \row^{+}(C_1,\ell_1)\bigcap R_{23}$. Given such $w^*$, we can construct a one-to-one mapping from $v\in \row(C_1,\ell_1)\bigcap R_{23}$ to $w\in \row^+(C_1,\ell_1)\bigcap R_{23}$ by $w=v+w^*$. Hence $\big|\row^{+}(C_1,\ell_1)\bigcap R_{23}\big|=\big|\row(C_1,\ell_1)\bigcap R_{23}\big|$ and we can conclude that $r(\ell_1+1,\ell_2,\ell_3)=r(\ell_1,\ell_2,\ell_3)+1$. In both cases, we have 
$$\ell_1+1+\ell_2+\ell_3-r(\ell_1+1,\ell_2,\ell_3)\geq \ell_1+\ell_2+\ell_3-r(\ell_1,\ell_2,\ell_3).$$
A similar reasoning applies to $r(\ell_1,\ell_2+1,\ell_3)$ and $r(\ell_1,\ell_2,\ell_3+1)$. Our conclusion then follows from induction.

In summary, we have shown that
$$\e\bigl(Z(\bsk_1)Z(\bsk_2)\bigr)\leq \overline{P}(\ell_1,\ell_2,\ell_3) =\begin{cases}
0 & \text{if } \min(\ell_1+\ell_2,\ell_1+\ell_3,\ell_2+\ell_3)\leq m-t,\\
    2^{-2(m-t-3)} & \text{if } \max(\ell_1,\ell_2,\ell_3)> m-t,\\
    2^{-2m+T+3} & \text{otherwise}.
\end{cases}$$
The sum of $\hat{g}(\bsk_1)\hat{g}(\bsk_2)\hat{g}(\bsk_{12})\e\bigl(Z(\bsk_1)Z(\bsk_2)\bigr)$ over all $|\kappa_{1j}|+|\kappa_{2j}|+|\kappa_{12j}|=2, j\in 1{:}3$ cases is bounded by
\begin{align*}
    &\phantom{\leq\,}\sum_{\ell_1=1}^\infty\sum_{\ell_2=1}^\infty\sum_{\ell_3=1}^\infty\frac{1}{27}2^{-2\ell_1-2\ell_2-2\ell_3-6} \overline{P}(\ell_1,\ell_2,\ell_3)\\
   & \leq  \frac{1}{27}\max(2^{2t},2^{T-3})2^{-2m} \bigg(\sum_{\ell=m-t+1}^\infty 2^{-\ell}\bigg)^3\\
   & =  \frac{1}{27}\max(2^{5t},2^{3t+T-3})2^{-5m},
\end{align*}
where the first inequality is equation~\eqref{eqn:cubicbound} specialized to this example. Similar bounds can be established for the sum over other cases. A calculation similar to equation~\eqref{eqn:varlowerbound} also shows that the variance of $\hat{\mu}$ is lower bounded by $c2^{-3m}$ for some $c>0$, and hence the skewness $\gamma$ is $O(\max(2^{5t},2^{3t+T})2^{-5m})$.

As we have mentioned, there exist $(1,m,3)$-digital nets such as Sobol' nets, so $t$ can be made small. The $T$ defined by equation~\eqref{eqn:Tdef} is, however, another characteristic of our generating matrices and not controlled by $t$. In this context, one can view Theorem~\ref{thm:Z1Z2bound} as showing that $T/m$ converges to $0$ in probability when $C_1,C_2,C_3$ are independently drawn from a uniform distribution on $\{0,1\}^{m\times m}$. A full exploration is beyond the scope of this paper and we leave it for future research.

\section{Discussion}\label{sec:disc}

We have given conditions under which random linear
scrambling of base $2$ digital nets leads
to small values of $\e((\hat\mu-\mu)^3)$
as $n=2^m\to\infty$. This numerator in 
the skewness $\gamma$ is generally 
$O(n^{-9/2+\epsilon})$ and is even
$O(n^{-5+\epsilon})$ under a condition on 
randomly chosen generator
matrices that holds with probability approaching one
as $n\to\infty$.  Our conditions on $f$ ensure
that $\e((\hat\mu-\mu)^2)=O(n^{-3+\epsilon})$. For an integrand $f$ with $\e((\hat\mu-\mu)^2)=\Omega(n^{-3})$ we get
$\gamma = O(n^\epsilon)$ for any $\epsilon>0$, which is
`almost $O(1)$'. Under the condition on randomly chosen
generator matrices, we get $\gamma = O(n^{-1/2+\epsilon})$,
although that rate includes a bound with possibly very slow convergence of $m^{3s}2^{-2\epsilon m}$ to zero when a small value of $\epsilon$ is chosen.

Our results are for smooth $f$, while the empirical
findings in \cite{ci4rqmc} included some non-smooth
integrands.
We used smoothness of $f$ in order to 
bound $\e((\hat\mu-\mu)^3)$, the numerator
of $\gamma$.  Specifically, $f$ had
to be continuously differentiable up to two times
in each variable. Continuous differentiability up to one
time in each variable puts those derivatives in $L^2[0,1]^s$ and hence makes the denominator $\e( (\hat\mu-\mu)^2)^{3/2}=O(n^{-9/2+\epsilon})$.  For discontinuous
or otherwise non-smooth $f$, the denominator will generally
be larger which then reduces $\gamma$.
If $f$ is once continuously differentiable in every
variable, but not twice continuously differentiable in
every variable, then the denominator of $\gamma$ is just
as small as before but the numerator may no longer 
satisfy our bounds. It is possible that integrands
of this intermediate smoothness level would have large
skewness.

The reliability of Student's $t$ based confidence
intervals 
in \cite{ci4rqmc} was not just for
this one kind of scrambled nets. It was evident for
nets scrambled as in \cite{rtms} as well as for some
randomized lattice rules. The skewness for those methods
raises different technical issues than the linearly
scrambled nets we considered do.

While \cite{ci4rqmc} show good empirical performance
when simply using the standard Student's $t$ method
to get asymptotic confidence
intervals for $\mu$ on based on replicates $\hat\mu_1,\dots,\hat\mu_R$,
much work remains in order to bound the coverage error
of those intervals.
Such an error analysis needs to consider both the central
limit theorem as $R\to\infty$ and the distribution of $\hat\mu$
as $n\to\infty$.  Classical Berry-Esseen bounds require
$\e( |\hat\mu-\mu|^3)/\e(|\hat\mu-\mu|^2)^{3/2}$ to be finite
but our theory does not have an absolute value in the
numerator.  The theory of Hall \cite{hall:1986,hall:1988} uses 
central moments of order at least six and we know that 
$\e( (\hat\mu-\mu)^4)/\e( (\hat\mu-\mu)^2)^2\to\infty$ is typical
for random linear scrambling as $n\to\infty$. So both Hall's theory and Berry-Esseen apply to a limit as $R\to\infty$ for fixed $n$.  This limits how well they can explain the results in \cite{ci4rqmc} which had large $n$.

The skewness of $\hat\mu$ is a centered and scaled
third moment. Therefore its sampling variance depends on the sixth moment of $\hat\mu$ and even an appropriately scaled fourth moment diverges to infinity as $n\to\infty$.  As a result, we believe that distinguishing empirically between $\gamma=O(n^\epsilon)$ and
$\gamma=O(n^{-1/2+\epsilon})$ would be extremely expensive.


In empirical work we have seen that the distribution
of $\hat\mu$ tends to become symmetric as $n$ increases.
We believe but have not proved 
that the distribution of $\hat\mu$ is asymptotically
symmetric as $n\to\infty$.  We know from \cite{superpolymulti} that for smooth integrands, the distribution of $\hat\mu$ has a small probability
of severe outliers which makes a median-of-means strategy very effective.
For confidence intervals, a data set $\hat\mu_1,\dots,\hat\mu_R$ with
no outliers would be expected to have reasonable coverage.  When there
are one or more outliers, the mean is adversely affected but the standard
deviation is also greatly enlarged making the Student's $t$ interval
wide enough that it can still cover the true mean $\mu$.  For instance,
suppose there are $R-1$ replicates $\hat\mu_r\approx0=\mu$ and one outlying replicate
at $c\ne0$. Without loss of generality we can take $c=1$. Then we get $\hat\mu \approx 1/{R}$ and the sample standard error is approximately
$$
\biggl[\frac1{R(R-1)}\Bigl[ (R-1)\Bigl(0-\frac1R\Bigr)^2+\Bigl(1-\frac1R\Bigr)^2\Bigr]\biggr]^{1/2}=\frac1R\sqrt{\frac{R}{R-1}}.
$$
Then the true mean is less than one standard error from the sample mean
and will be well within the Student's  $t$ confidence intervals for $95$\%
coverage which always extend at least $1.96$ standard errors in each direction from the sample mean. An explanation of this type does not require a symmetric distribution for $\hat\mu_r-\mu$. It is enough for that distribution to be a mixture of a benign distribution with probability $1-\epsilon$ and a symmetric outlier distribution with probability $\epsilon$. The benign component need not be symmetric.

\section*{Acknowledgments}
This work was supported by the U.S.\ National Science Foundation
under grant under grant DMS-2152780 and by a Stanford Graduate Fellowship in Science \& Engineering.

\bibliographystyle{plain}
\bibliography{qmc}

\begin{thebibliography}{10}

\bibitem{cafmowen}
R.~E. Caflisch, W.~Morokoff, and A.~B. Owen.
\newblock Valuation of mortgage backed securities using {Brownian} bridges to reduce effective dimension.
\newblock {\em Journal of Computational Finance}, 1:27--46, 1997.

\bibitem{dick:pill:2010}
J.~Dick and F.~Pillichshammer.
\newblock {\em Digital sequences, discrepancy and quasi-{Monte Carlo} integration}.
\newblock Cambridge University Press, Cambridge, 2010.

\bibitem{dick:2009}
Josef Dick.
\newblock The decay of the {Walsh} coefficients of smooth functions.
\newblock {\em Bulletin of the Australian Mathematical Society}, 80(3):430--453, 2009.

\bibitem{gran:boyd:ye:2006}
Michael Grant, Stephen Boyd, and Yinyu Ye.
\newblock Disciplined convex programming.
\newblock {\em Global optimization: From theory to implementation}, pages 155--210, 2006.

\bibitem{hall:1986}
P.~G. Hall.
\newblock On the bootstrap and confidence intervals.
\newblock {\em The Annals of Statistics}, pages 1431--1452, 1986.

\bibitem{hall:1988}
P.~G. Hall.
\newblock Theoretical comparisons of bootstrap confidence intervals.
\newblock {\em The Annals of Statistics}, 16(3):927--953, 1988.

\bibitem{joe:kuo:2008}
S.~Joe and F.~Y. Kuo.
\newblock Constructing {Sobol'} sequences with better two-dimensional projections.
\newblock {\em SIAM Journal on Scientific Computing}, 30(5):2635--2654, 2008.

\bibitem{larc:nied:schm:1996}
G.~Larcher, H.~Niederreiter, and W.~{Ch.} Schmid.
\newblock Digital nets and sequences constructed over finite rings and their application to {quasi-Monte Carlo} integration.
\newblock {\em Monatshefte f{\"u}r Mathematik}, 121:231--253, 1996.

\bibitem{ci4rqmc}
Pierre L'Ecuyer, Marvin~K. Nakayama, Art~B. Owen, and Bruno Tuffin.
\newblock Confidence intervals for randomized quasi-monte carlo estimators.
\newblock Technical report, hal-04088085, 2023.

\bibitem{mato:1998:2}
J.~Matou\v{s}ek.
\newblock On the {L$^2$}--discrepancy for anchored boxes.
\newblock {\em Journal of Complexity}, 14:527--556, 1998.

\bibitem{rtms}
A.~B. Owen.
\newblock Randomly permuted $(t,m,s)$-nets and $(t,s)$-sequences.
\newblock In H.~Niederreiter and P.~J.-S. Shiue, editors, {\em Monte Carlo and Quasi-Monte Carlo Methods in Scientific Computing}, pages 299--317, New York, 1995. Springer-Verlag.

\bibitem{smoovar}
A.~B. Owen.
\newblock Scrambled net variance for integrals of smooth functions.
\newblock {\em Annals of Statistics}, 25(4):1541--1562, 1997.

\bibitem{localanti}
A.~B. Owen.
\newblock Local antithetic sampling with scrambled nets.
\newblock {\em Annals of Statistics}, 36(5):2319--2343, 2008.

\bibitem{nonzerogain}
Z.~Pan and A.~B. Owen.
\newblock The nonzero gain coefficients of {Sobol's} sequences are always powers of two.
\newblock {\em Journal of Complexity}, 75:101700, 2023.

\bibitem{superpolyone}
Z.~Pan and A.~B. Owen.
\newblock Super-polynomial accuracy of one dimensional randomized nets using the median-of-means.
\newblock {\em Mathematics of Computation}, 92(340):805--837, 2023.

\bibitem{superpolymulti}
Z.~Pan and A.~B. Owen.
\newblock Super-polynomial accuracy of multidimensional randomized nets using the median-of-means.
\newblock {\em Mathematics of Computation}, 2024.

\bibitem{sobol67}
I.~M. Sobol'.
\newblock The distribution of points in a cube and the accurate evaluation of integrals (in {R}ussian).
\newblock {\em Zh. Vychisl. Mat. i Mat. Phys.}, 7:784--802, 1967.

\bibitem{yosh:2017}
Takehito Yoshiki.
\newblock Bounds on {Walsh} coefficients by dyadic difference and a new {Koksma-Hlawka} type inequality for quasi-{Monte Carlo} integration.
\newblock {\em Hiroshima Mathematical Journal}, 47(2):155--179, 2017.

\bibitem{YueHic02a}
R.~X. Yue and F.~J. Hickernell.
\newblock The discrepancy and gain coefficients of scrambled digital nets.
\newblock {\em Journal of Complexity}, 18:135--151, 2002.

\end{thebibliography}
\end{document}